\documentclass[letterpaper,11 pt]{article}

\usepackage{amsmath, amssymb, bbm, xspace}

\def\spose#1{\hbox to 0pt{#1\hss}}

\def\text #1{\hbox{\quad#1\quad}}

\def\nthinsp{\mskip -2   mu}

\def\superstar{^{\raise 0.5pt\hbox{$\nthinsp *$}}}
\def\SUPERSTAR{^{\raise 0.5pt\hbox{$*$}}}

\def\lamstarT {\lambda^{\raise 0.5pt\hbox{$\nthinsp *$}T}}

\def\hbar{\skew{4.2}\bar h}

		\def\bkE{{\rm I\kern-.17em E}}
		\def\bk1{{\rm 1\kern-.17em l}}
		\def\bkD{{\rm I\kern-.17em D}}
		\def\bkR{{\rm I\kern-.17em R}}
		\def\bkP{{\rm I\kern-.17em P}}
		\def\bkY{{\bf \kern-.17em Y}}
		\def\bkZ{{\bf \kern-.17em Z}}

		\def\beq{\begin{eqnarray}}
		\def\bc{\begin{center}}
		\def\be{\begin{enumerate}}
		\def\bi{\begin{itemize}}
		\def\bs{\begin{small}}
		\def\bS{\begin{slide}}
		\def\ec{\end{center}}
		\def\ee{\end{enumerate}}
		\def\ei{\end{itemize}}
		\def\es{\end{small}}
		\def\eS{\end{slide}}
		\def\eeq{\end{eqnarray}}
		\def\qed{\quad \vrule height7.5pt width4.17pt depth0pt}

	\def\cp2problem#1#2#3#4{\fbox
		 {\begin{tabular*}{0.9\textwidth}
			{@{}l@{\extracolsep{\fill}}l@{\extracolsep{6pt}}l@{\extracolsep{\fill}}c@{}}
				#1 & & $#4 $ 
			\end{tabular*}}}

		\renewcommand{\emph}[1]{\textbf{#1}}

		\def\bkE{{\rm I\kern-.17em E}}
		\def\bk1{{\rm 1\kern-.17em l}}
		\def\bkD{{\rm I\kern-.17em D}}
		\def\bkR{\mathbb{R}}
		\def\bkP{{\rm I\kern-.17em P}}
		
		\def\bkZ{{\bf{Z}}}

\newcommand {\beeq}[1]{\begin{equation}\label{#1}}
\newcommand {\eeeq}{\end{equation}}
\newcommand {\bea}{\begin{eqnarray}}
\newcommand {\eea}{\end{eqnarray}}

\def\texitem#1{\par\smallskip\noindent\hangindent 25pt
               \hbox to 25pt {\hss #1 ~}\ignorespaces}

\usepackage{graphicx}
\usepackage{amsfonts,amsmath,fullpage,bbm}
\usepackage{amssymb,amsthm,multirow,verbatim}
\usepackage{acronym,wrapfig,plain,mathrsfs,enumerate,relsize,color}
\newtheorem{algorithm}{Algorithm}
\usepackage{algorithm}
\usepackage{algorithm,algorithmic}
\newtheorem{theorem}{Theorem}
\newtheorem{remark}{Remark}
\newtheorem{corollary}{Corollary}
\newtheorem{lemma}{Lemma}

\newtheorem{proposition}{Proposition}
\newtheorem{assumption}{Assumption}
\newtheorem{definition}{Definition}
\usepackage{subfig}
\usepackage{algorithmic}
\usepackage{pifont}
\usepackage{footmisc}
\usepackage{stackengine}

\usepackage[pdftex,colorlinks=true,urlcolor=blue,citecolor=black,anchorcolor=black,linkcolor=black]{hyperref}

\usepackage{algorithmic}

\newcommand{\zi}[1]{{\color{black}#1}}

\usepackage[normalem]{ulem}

\begin{document}

\begin{center}
  {\Large \textbf{Regularized Projection Algorithms for Monotone Inverse Variational Inequalities}}\\[4pt]
  {\large Griffin Smith, Zeinab Alizadeh, Afrooz Jalilzadeh}
\end{center}
\vspace{0.5em}

\section*{Abstract}
Stochastic inverse variational inequalities (SIVIs) arise in applications in which equilibrium responses are observed under uncertainty, such as inverse road pricing and network equilibrium control. Existing methods typically rely on co-coercivity or strong monotonicity, while general monotone SIVIs remain less understood. We propose a regularized projection algorithm that combines Tikhonov regularization with increasing batch sizes. Under monotonicity and Lipschitz continuity, we prove almost sure boundedness of the iterates and almost sure convergence of their distance to the SIVI solution set. We further establish, to the best of our knowledge, the first explicit nonasymptotic rate of $O(T^{-1/2})$ for the expected squared residual under general monotonicity. This yields $O(\epsilon^{-2})$ iterations and $O(\epsilon^{-4-2\delta})$ stochastic oracle calls, for any $\delta>0$, to obtain an $\epsilon$-solution in expected squared residual. A deterministic variant attains the same iteration complexity using $O(\epsilon^{-2})$ exact operator evaluations. Numerical experiments illustrate the proposed methods on monotone SIVI problems.

\section{Introduction}

Interest in Inverse Variational Inequalities (IVIs) has grown significantly in recent decades \cite{he1999goldstein}, driven by their broad applicability across diverse domains such as transportation system operation, the formulation of control policies, and electrical power network management \cite{yang1997traffic,he2011inverse,scrimali2012inverse,he2018existence}. 
In this paper, we focus on Stochastic Inverse Variational Inequalities (SIVI). In particular, let $X\subseteq \mathbb R^n$ be a nonempty closed convex set, and $F:\mathbb R^n \to \mathbb R^n$ be a continuous nonlinear map. In the SIVI problem, we seek  $x^* \in \mathbb R^n$ such that $F(x^*) \in X$ and
\begin{align}\label{sIvI}\tag{SIVI}
\langle y-F(x^*), x^* \rangle\geq 0, \quad \forall y\in X,
\end{align}
where $F(x)\triangleq \mathbb E[G(x,\xi)]$, $\xi: \Omega \to \mathbb R^d$, ${G}: \mathbb R^n \times \mathbb R^d  \rightarrow
\mathbb{R}^n$, and the associated probability space is denoted by $(\Omega, {\cal F}, \mathbb{P})$. The SIVI framework seeks a solution $x^*$ that satisfies a specific system of inequalities across the entire feasible set $X$. When $F^{-1}$ is single-valued and well defined on $X$, the  SIVI is mathematically equivalent to a Stochastic Variational Inequality (SVI) defined by $H \triangleq F^{-1}$; however, the inverse mapping is often unavailable or computationally expensive to evaluate, necessitating the development of ``inverse-free" iterative schemes. While VIs \cite{YuriiNesterov2011DiscreteandContinuousDynamicalSystems,facchinei2007finite,malitsky2015projected,alizadeh2026variance} and SVIs \cite{jiang2008stochastic,koshal2012regularized,yousefian2017smoothing,jalilzadeh2019proximal,alizadeh2024randomized}  have been studied extensively, the study of IVIs and SIVIs is still underdeveloped. 
In the deterministic setting, approximation methods have been developed for IVIs under assumptions such as strong monotonicity and co-coercivity. These approaches include proximal point-based algorithms, projection-type methods, alternating contraction projection methods, and dynamical system methods \cite{vuong2021global,dey2024dynamical,xu2021notes,chen2020tikhonov,luo2014regularization}. Specifically, projection-based algorithms were pioneered by \cite{he2011inverse,he2018existence,luo2014regularization}, and \cite{zou2016novel}  introduced neural network-based methods to approximate solutions.
To ensure stability in deterministic and monotone settings without stronger structural assumptions such as strong monotonicity or co-coercivity, Anh and Hai \cite{anh2024regularized} extended the regularized iterative framework of \cite{luo2014tikhonov} by proposing a regularized dynamical system for monotone IVIs in Hilbert spaces. Their approach admits a unique strong global solution and establishes strong convergence under relatively mild assumptions. However, despite these stability and asymptotic convergence guarantees, explicit nonasymptotic convergence rate results remain unavailable.
In the stochastic domain, research is even more sparse. In \cite{alizadeh2023projection}, a variance-controlled projection-based algorithm was developed specifically for SIVIs when the operator is co-coercive. They established almost sure convergence of the generated iterates to the solution and provided a convergence rate guarantee of $\mathcal{O}(1/T)$, which stands as a benchmark for the co-coercive class of problems. However, co-coercivity is a restrictive assumption that may not hold for general monotone mappings. Despite these advancements, an explicit convergence rate for the general monotone SIVIs remains unavailable in the current literature to the best of our knowledge.

\subsection{Applications}\label{application}
IVIs arise in transportation, network economics, power systems, and other equilibrium-control problems, where the goal is to choose a control vector whose induced equilibrium response satisfies prescribed constraints. In stochastic settings, this response is represented by the mean map $F(x):=\mathbb{E}[G(x,\xi)]$. We next present two representative SIVI formulations.

\textbf{Inverse Road Pricing Problem:}
Road pricing is a classical traffic-management tool used to regulate congestion and achieve operational or environmental targets \cite{yang1997traffic}. In a standard traffic-equilibrium problem, the toll vector is fixed and the corresponding user-equilibrium flow is determined according to Wardrop's principle. In contrast, an inverse road-pricing problem seeks a toll vector that induces a desirable equilibrium flow pattern. This interpretation is consistent with the black-box IVI framework in \cite{he2011inverse}, in which the toll-to-flow response can be computed approximately even when its inverse is unavailable.

Consider a network with $n$ controlled links, and let $x\in\mathbb{R}^n$ denote the vector of link tolls. For a given toll vector $x$, let $a^*(x)\in\mathbb{R}^n$ denote the induced equilibrium link-flow vector. The set of admissible flow patterns may be defined as $X:=\{a\in\mathbb{R}^n\mid \underline a\leq Ca\leq \bar a\}$, where $C\in\mathbb{R}^{q\times n}$ is a selection or aggregation matrix and $\underline a,\bar a\in\mathbb{R}^q$ specify flow, capacity, or environmental bounds. The stochastic inverse road-pricing problem can then be modeled as finding $x^*\in\mathbb{R}^n$ such that $F(x^*)\in X$ and $\langle y-F(x^*),x^*\rangle\geq 0,$ for all $y\in X$.
Here $G(x,\xi)$ denotes the stochastic equilibrium-flow response under toll vector $x$, where $\xi$ may represent random demand, incidents, weather, or measurement errors. For additive zero-mean measurement noise, one may write $G(x,\xi)=a^*(x)+\xi$, in which case $F(x)=a^*(x)$. Since the toll-to-flow mapping is generally available only through observations, simulations, or equilibrium computations, this formulation motivates stochastic inverse-free methods.

\textbf{Resource-Constrained Network Equilibrium:}
Consider a commodity network with $m$ supply markets, $n$ demand markets, and $mn$ shipment routes. Let $a\in\mathbb{R}^{mn}_+$ denote the shipment vector and let $x=[x_1^\top,x_2^\top]^\top\in\mathbb{R}^{m+n}$ denote a policy vector, where $x_1\in\mathbb{R}^m$ and $x_2\in\mathbb{R}^n$ represent supply- and demand-side tax adjustments. For a fixed policy $x$, the equilibrium shipment vector $a^*(x)$ is obtained from the parametric VI
\begin{align}\label{eq:mapping_app}
\left\langle a'-a,\,
\ell(a)+A^\top\bigl(g(Aa)+\alpha+x_1\bigr)
-B^\top\bigl(h(Ba)-\beta-x_2\bigr)
\right\rangle\geq 0 ,
\end{align}
for all $a'\in\mathbb{R}^{mn}_+$, where $A:=I_m\otimes e_n^\top\in\mathbb{R}^{m\times mn}$ and $B:=e_m^\top\otimes I_n\in\mathbb{R}^{n\times mn}$ aggregate shipments into supply and demand vectors. The mappings $\ell$, $g$, and $h$ represent transportation costs, supply prices, and demand prices, respectively, while $\alpha$ and $\beta$ denote baseline policy parameters.
Define the induced aggregate response by
$F(x):=[(Aa^*(x))^\top,(Ba^*(x))^\top]^\top\in\mathbb{R}^{m+n}$, and let $X:=\{u\in\mathbb{R}^{m+n}\mid Lu\leq b\}$ denote the set of admissible aggregate supply-demand states. Following the constrained black-box IVI framework of \cite{he2010solving}, the regulator seeks $x^*\in\mathbb{R}^{m+n}$ such that $F(x^*)\in X$ and $\langle v-F(x^*),x^*\rangle\geq 0,$ for all $v\in X$.
Under uncertainty, the regulator observes a stochastic response $G(x,\xi)$ satisfying $\mathbb{E}[G(x,\xi)]=F(x)$. For example, under additive zero-mean observation noise,
$G(x,\xi):=[(Aa^*(x)+\xi_1)^\top,(Ba^*(x)+\xi_2)^\top]^\top$, where $\xi_1\in\mathbb{R}^m$ and $\xi_2\in\mathbb{R}^n$. Such network models provide a natural application of the proposed method whenever the resulting mean response is monotone and Lipschitz continuous but is not assumed to be co-coercive \cite{nagurney2013network}.

These examples involve controlling an equilibrium response observed through noisy data, simulations, or equilibrium computations. Since the inverse map may be unavailable and the mean response may be monotone but not co-coercive or strongly monotone, they motivate methods with guarantees for monotone SIVIs.

\subsection{Contribution}
Motivated by the lack of explicit convergence-rate guarantees for stochastic IVIs under general monotonicity, we propose the Regularized Variance-Controlled Inverse Projected Gradient (RVC-IPG) method, which combines Tikhonov regularization with increasing batch sizes to handle monotonicity and stochastic error without requiring co-coercivity or strong monotonicity.
The closest related works are \cite{anh2024regularized} and \cite{alizadeh2023projection}. Anh and Hai \cite{anh2024regularized} study deterministic monotone IVIs and establish asymptotic convergence, but no explicit nonasymptotic rate is provided. Our previous work \cite{alizadeh2023projection} proves an $\mathcal O(1/T)$ rate for stochastic IVIs under co-coercivity; in contrast, the present paper removes co-coercivity and establishes an explicit $\mathcal O(T^{-1/2})$ residual rate under general monotonicity. Our contributions are threefold:
\begin{enumerate}
\item[(i)] We develop an almost sure convergence theory for monotone SIVIs. We show that each regularized SIVI has a unique solution, the regularization path is bounded, and its limit points solve the SIVI. Building on these properties, we prove that the RVC-IPG iterates are almost surely bounded and their distance to the SIVI solution set converges to zero almost surely.

\item[(ii)] We establish an explicit nonasymptotic residual bound. For suitable choices of the stepsize, regularization, projection parameter, and batch size, we prove that $\min_{0\le k\le T-1}\mathbb E[|H(x_k,\eta_k)|^2]=O(T^{-1/2})$, where $H$ is the natural residual map associated with the inverse projection condition. This yields an $O(\epsilon^{-2})$ iteration complexity and an $O(\epsilon^{-4-2\delta})$ stochastic oracle complexity, for any $\delta>0$, to obtain an $\epsilon$-solution in expected squared residual. To the best of our knowledge, this is the first explicit nonasymptotic rate for stochastic inverse variational inequalities under general monotonicity, without co-coercivity or strong monotonicity of the mean operator.

\item[(iii)] We connect the rate result to practical residual evaluation and deterministic computation. We show that the same $O(T^{-1/2})$ rate holds for a fixed projection parameter, enabling consistent comparison across iterations in numerical experiments. We also analyze a deterministic counterpart of RVC-IPG and prove that it attains the same $O(\epsilon^{-2})$ iteration complexity using exact evaluations of the mean operator. Numerical experiments on SIVI instances that are monotone but not co-coercive illustrate the behavior of the proposed methods and support the theoretical findings.
\end{enumerate}


\section{Preliminaries}\label{Preliminaries}
In this section, we first introduce notation and then state the main assumptions used in the convergence analysis.

{\bf Notation.}  Throughout the paper, $\|x\|$ denotes the Euclidean vector norm, i.e., $\|x\|=\sqrt{x^Tx}$.  $\mathbf{P}_X [u]$ is the projection of $u$ onto the set $X$, i.e. $\mathbf{P}_X [u] = \mbox{argmin}_{z \in X} \frac{1}{2}\| z-u \|^2$. $\mathbb E[x]$ is used to denote the expectation of a random variable $x$. For $\rho>0$, $B_\rho(x)$ denotes the closed ball of radius $\rho$ centered at $x$, i.e.,
$B_\rho(x)\triangleq \{z\in\mathbb R^n:\|z-x\|\le \rho\}$.

\begin{definition}
     An operator $A: X \rightarrow \mathbb{R}^n$ is said to be 
   
\noindent (i) $\mu$-strongly monotone if and only if there exists a constant $\mu > 0$ such that
$\langle A(x) - A(y), x - y\rangle \geq \mu \|x - y\|^2$ for all $x, y \in X.$

\noindent (ii)  monotone if and only if $
\langle A(x) - A(y), x - y\rangle \geq 0$ for all $x, y \in X.$

\noindent (iii) $L$-Lipschitz continuous if and only if  there exists a constant $L > 0$ such that 
$\|A(x) - A(y)\| \leq L \|x - y\|,$ for all $x, y \in X.$

\end{definition}

Define $\mathcal F_{k}$ denotes the information history, i.e., $\mathcal F_k := \sigma\{x_0,\xi_{j,s}:0\le s\le k-1;1\le j\le N_s\},$ at epoch $k$. We impose the following assumption on the stochastic error.
\begin{align}\label{def:w}
\bar w_{k,N_k} \triangleq \tfrac{1}{N_k}{\sum_{j=1}^{N_k} ( G(x_k,\xi_{j,k})-F(x_k))}
\end{align}
\begin{assumption}\label{assump_error}
 There exists $\nu>0$ such that $\mathbb E[\bar w_{k,N_{k}}\mid  \mathcal F_{k}]=0$ and $\mathbb E[\| \bar w_{k,N_{k}}\|^2\mid  \mathcal F_{k}]\leq \tfrac{\nu^2}{N_{k}}$  holds almost surely 
 for all $k$.  
\end{assumption} 

In our analysis, the following technical lemmas are used.
\begin{lemma}\label{proj}
\noindent Let $ X\subseteq \mathbb{R}^n $  be a nonempty closed and convex set. Then, 
(i) $\|\mathbf{P}_X [u]- \mathbf{P}_X [v]\| \leq \|u-v\| $ for all $ u,v \in \mathbb{R}^n$;
(ii) $ (u-\mathbf{P}_X [u])^T(x-\mathbf{P} _X [u]) \leq 0 $ for all $u \in \mathbb{R}^n$ and $x \in X$.
\end{lemma}

\begin{lemma}\label{IVIunqiue}
(\cite{he2018existence}, Theorem 3.2). 
Let \(X\subseteq \mathbb{R}^n\) be nonempty, closed, and convex, and suppose
that the mean operator \(F(x):=\mathbb{E}[G(x,\xi)]\) is well-defined,
Lipschitz continuous, and strongly monotone. Then the stochastic inverse
variational inequality \((\mathrm{SIVI})\) has a unique solution.
\end{lemma}


\begin{lemma}\label{MVT}
    Let $h(t): [0, \infty) \to \mathbb{R}$ be a continuously differentiable, strictly decreasing, and convex function. Define the sequence $h_k = h(k)$ for $k \in \mathbb{N}$. Then, for every $k$, there exists a $\zeta \in (k, k+1)$ such that
    $\left| h_k - h_{k+1} \right|^2 = \left| h'(\zeta) \right|^2$.
    Furthermore, if $|h'(t)|$ is non-increasing, then
    $\left| h_k - h_{k+1} \right|^2 \leq \left| h'(k) \right|^2$.
\end{lemma}

\begin{proof}
    Since $h(t)$ is continuously differentiable on the interval $[k, k+1]$, the mean value theorem asserts the existence of a point $\zeta \in (k, k+1)$ such that:
    $h(k+1) - h(k) = h'(\zeta)((k+1) - k).$
    Taking the absolute value of both sides and squaring yields
    $|h_k - h_{k+1}|^2 = |h'(\zeta)|^2.$
    Since $h(t)$ is strictly decreasing and convex, it follows that $h'(t)$ is a non-decreasing function. Consequently, $|h'(t)|$ is non-increasing for $t \geq 0$. Thus, for any $\zeta$ in the interval $[k, k+1]$,
    $|h_k - h_{k+1}|^2 = |h'(\zeta)|^2 \leq \max_{t \in [k, k+1]} |h'(t)|^2 = |h'(k)|^2.$ 
\end{proof}
We will also use the following standard almost-supermartingale convergence result to establish the almost sure stability and asymptotic consistency of the stochastic iterates.
\begin{lemma}[Robbins--Siegmund \cite{robbins1971convergence}]
\label{lem:RS}
Let $\{v_k\}$, $\{u_k\}$, $\{\alpha_k\}$, and $\{\beta_k\}$ be nonnegative
random variables adapted to the filtration $\{\mathcal F_k\}$. Suppose that
$\mathbb{E}[v_{k+1}\mid \mathcal F_k]
\leq (1+\alpha_k)v_k-u_k+\beta_k$ almost surely for all  $k\geq 0,$
and that
$\sum_{k=0}^{\infty}\alpha_k<\infty,$
$\sum_{k=0}^{\infty}\beta_k<\infty$ almost surely.
Then $\{v_k\}$ converges almost surely to a nonnegative random variable $v$, and
$\sum_{k=0}^{\infty}u_k<\infty$
almost surely.
\end{lemma}
The next lemma establishes the closedness of the SIVI solution set, a property used later in the convergence analysis.
\begin{lemma} \label{closedSol}
Let $X \subseteq \mathbb{R}^n$ be closed and $F: \mathbb{R}^n \to \mathbb{R}^n$ be continuous. 
Then the SIVI solution set
$S \triangleq \{x \in \mathbb{R}^n : F(x) \in X, \langle y - F(x),x \rangle \geq 0,\ \forall y \in X\}$
is closed.
\end{lemma}

\begin{proof}
Let $\{x_k\}$ be a sequence in $S$ that converges to $\bar x$. 
Since each $x_k \in S$, we have $F(x_k) \in X$ for all $k \in \mathbb{N}$. 
Since $F$ is continuous, $F(x_k)$ converges to $F(\bar x)$; because $X$ is closed, $F(\bar x) \in X$. 
Moreover, for any $y \in X$, $\langle y - F(x_k), x_k \rangle \geq 0$.
    Taking the limit as $k \to \infty$ and noting the continuity of $F$ and the inner product gives
    $\langle y - F(\bar x), \bar x \rangle \geq 0.$
    As this inequality holds for all $y \in X$ and $F(\bar x) \in X$ it follows that $\bar x \in S$.
    Therefore every convergent sequence in $S$ has its limit in $S$, and $S$ is closed.
\end{proof}

\section{Proposed Method}\label{proposed}
In this section, we present the Regularized Variance-Controlled Inverse Projected Gradient (RVC-IPG) method for solving the monotone \eqref{sIvI} problem. First, we show that solving  \eqref{sIvI}  is equivalent to finding $x^*\in \mathbb{R}^n$ such that
\begin{equation}\label{opt-cond}
    F(x^*)=\mathbf{P}_X(F(x^*)-\eta x^*),\quad \text{where}\quad \eta>0.
\end{equation} 
\begin{proposition}
    $x^*$ is a solution of problem \eqref{sIvI} if and only if $F(x^*)=\mathbf{P}_X(F(x^*)-\eta x^*)$.   
\end{proposition}
\begin{proof}
    We start by rewriting the projection equation \eqref{opt-cond}, which is equivalent to 
    $F(x^*)\in \mbox{argmin}_{y\in X}\|y-(F(x^*)-\eta x^*)\|^2$. We observe that $F(x^*)\in X$ and, since the objective function is convex, the first-order optimality condition is equivalent to finding a global solution. Hence,
    $$ \langle y-F(x^*),F(x^*)-(F(x^*)-\eta x^*) \rangle \geq 0
    \quad \text{for all } y\in X.$$
    Since $F(x^*)-(F(x^*)-\eta x^*)=\eta x^*$ and $\eta>0$, this is equivalent to
    $$\langle y-F(x^*), x^* \rangle \geq 0
    \quad \text{for all } y\in X,$$
    with $F(x^*)\in X$, which is exactly problem \eqref{sIvI}.
    
\end{proof}
Our approach is built upon two core principles. To address the ill-posedness of the monotone SIVI, we introduce a Tikhonov-type regularization. Let $\Phi: \mathbb{R}^n \to \mathbb{R}^n$ be a $\mu$-strongly monotone and  Lipschitz continuous operator. For a given regularization parameter $\lambda > 0$, we define the regularized operator as
$F_{\lambda}(x) \triangleq F(x) + \lambda \Phi(x).$
Second, to manage the stochastic uncertainty, we employ a variance-controlled strategy through a variable batch-size scheme. Rather than using a fixed number of samples, we utilize a non-decreasing sequence $\{N_k\}$ that increases the precision of the operator estimate as the iterates progress. Specifically, at each iteration $k$, we approximate the deterministic regularized operator by using the stochastic estimator $\frac{\sum_{j=1}^{N_k} G(x_k, \xi_{j,k})}{N_k}  + \lambda_k \Phi(x_k)$. 
The proposed Regularized Variance-Controlled Inverse Projected Gradient (RVC-IPG) method is displayed in Algorithm \ref{alg1}. Moreover, to measure how far the iterates are from the optimal solution,
we can examine the degree to which the optimality condition in \eqref{opt-cond} is violated. To this end, we define the residual map $H: \mathbb R^n \times \mathbb{R}_{++}\to \mathbb{R}^n$, such that $ H(x, \eta) \triangleq   F(x) - \mathbf{P}_X(F(x) - \eta x),
$ which will be used to analyze the convergence rate of the proposed method.

\begin{algorithm}[H]\normalsize
\caption{Regularized Variance-Controlled Inverse Projected Gradient (RVC-IPG) method}
\label{alg1}
\begin{algorithmic}
\STATE {\bf Input}: $x_0 \in \mathbb R^n$, $\{\eta_{k}, \theta_k, \lambda_k\}_{k\geq 0} \subseteq \mathbb{R}_+$ and $\{N_k\}_{k\ge0}\subseteq \mathbb N$;
\FOR{$k=0, \dots, T-1$}
\STATE  $z_k = \mathbf{P}_{X}\left[ \frac{\sum_{j=1}^{N_k} G(x_k, \xi_{j,k})}{N_k}  + \lambda_k \Phi(x_k) - \eta_{k} x_k\right]$
\STATE  $x_{k+1} = x_k - \theta_k \left( \frac{\sum_{j=1}^{N_k} G(x_k, \xi_{j,k})}{N_k}  + \lambda_k \Phi(x_k) - z_k \right)$
\ENDFOR
\end{algorithmic}
\end{algorithm}\vspace{-1cm}
\section{Convergence Analysis}\label{conv}
To obtain the convergence results of Algorithm \ref{alg1} we first establish the following assumptions and technical lemmas. 
\begin{assumption}\label{assum:problem}
Consider problem \eqref{sIvI}. Let the following hold.\\
    \noindent {(i)} $F:\mathbb R^n\to\mathbb R^n$ is monotone and $L_1$-Lipschitz continuous, \\
     \noindent {(ii)} $\Phi:\mathbb R^n\to\mathbb R^n$ is $\mu$-strongly monotone,
and $L_2$-Lipschitz continuous.\\
\noindent{(iii)} The solution set is nonempty, i.e., $S\triangleq\{x\in\mathbb{R}^n: F(x)\in X,\langle y-F(x),x\rangle\ge0\ \forall y\in X\}\ne\emptyset.$
\end{assumption}
For any mapping $A:\mathbb R^n\to\mathbb R^n$ and any nonempty closed convex set $X\subseteq\mathbb R^n$, we denote by $\mathrm{SIVI}(A,X)$ the inverse variational inequality of finding $x\in\mathbb R^n$ such that $A(x)\in X$ and
$\langle y-A(x),x\rangle\ge 0$ for all $y\in X$.
\begin{remark}\label{rem:regularized_sol}
For any regularization parameter $\lambda > 0$, denote $F_\lambda \triangleq F + \lambda \Phi$. Given that $F$ is monotone and $\Phi$ is $\mu$-strongly monotone, it follows that $F_\lambda$ is $\lambda\mu$-strongly monotone and $(L_1 + \lambda L_2)$-Lipschitz continuous. Consequently, Lemma \ref{IVIunqiue} guarantees that the regularized inverse variational inequality $\mathrm{SIVI}(F_\lambda, X)$ possesses a unique solution, which we denote by $x_\lambda$.
\end{remark}

The following regularization-path estimates are inspired by the deterministic
regularization analysis of Anh and Hai~\cite{anh2024regularized}, where a strongly monotone
operator is used to regularize monotone IVIs. Here we adapt these estimates
to the SIVI setting and use them to analyze a stochastic
projection method with explicit nonasymptotic residual bounds.

\begin{lemma}\label{M bound}
 The following statements hold. 
 \noindent (i) The sequence $\{x_{\lambda_k}\}_{k\geq 1}$ is bounded. 
 \noindent (ii) There exists an $M > 0$ such that for all $\lambda_k,\lambda_l > 0$ we have
$        \| x_{\lambda_k} - x_{\lambda_l} \| \leq M \tfrac{|\lambda_k - \lambda_l|}{\lambda_l}.$ 
\end{lemma}
\begin{proof}
    (i) Let $x^* \in S$. Since $x_\lambda$ is the unique solution of the $\lambda \mu$-strongly monotone $\mathrm{SIVI}(F_\lambda,X)$, for all $y \in X$ we have 
    \begin{align*}
        \lambda\mu\|x^*-x_\lambda\|^2 &\leq \langle F_\lambda(x^*)-F_\lambda(x_\lambda), x^*-x_\lambda \rangle \\ &= \langle F_\lambda(x^*)-y, x^*-x_\lambda \rangle + \langle y-F_\lambda(x_\lambda), x^* \rangle - \langle y-F_\lambda(x_\lambda), x_\lambda \rangle\\ &\leq \langle F_\lambda(x^*)-y, x^*-x_\lambda \rangle + \langle y-F_\lambda(x_\lambda), x^* \rangle , 
    \end{align*}
    where the last inequality is obtained because $x_\lambda$ is the solution of the $\mathrm{SIVI}(F_\lambda,X)$.
   Setting $y = F(x^*)$ yields
\begin{align*}
     \lambda \mu \|x^* - x_\lambda\|^2  \leq \lambda \langle \Phi (x^*),x^* - x_\lambda \rangle + \langle F(x^*) - F_\lambda (x_\lambda), x^* \rangle . 
\end{align*}
Note that by definition $F_\lambda(x_\lambda) \in X$ and taking $y = F_\lambda(x_\lambda)$ in the original $\mathrm{SIVI}(F,X)$ problem yields $\langle F_\lambda(x_\lambda) - F(x^*),x^* \rangle \geq 0$.
Thus $\langle F(x^*) - F_\lambda (x_\lambda), {x^*}\rangle\leq 0$ and using Cauchy-Schwarz, we obtain
    \begin{align}\label{bound phi}
         \mu \|x^* - x_\lambda\|^2 &\leq \langle \Phi (x^*), {x^* - x_\lambda}\rangle  \leq \|{\Phi (x^*)} \| \|x^* - x_\lambda\|,
    \end{align}
    which gives 
$   \|x^* - x_\lambda\| \leq  \tfrac{1}{\mu} \|\Phi (x^*)\|$.
Thus, for all $\lambda>0$, $x_\lambda \in B_\rho(x^*)$, where $\rho = \tfrac{1}{\mu}\|\Phi(x^*)\|$. In particular, the sequence $\{x_{\lambda_k}\}_{k\geq 1}$ is bounded.   

(ii) Using the definition of $x_{\lambda_k}$ we have
$$    \langle{F_{\lambda_{l}}(x_{\lambda_{l}}) - F_{\lambda_{k}}(x_{\lambda_k})},{x_{\lambda_k}} \rangle\geq 0,
    \quad \text{and} \quad
    \langle{F_{\lambda_{k}}(x_{\lambda_{k}}) - F_{\lambda_{l}}(x_{\lambda_{l}})},{x_{\lambda_{l}}}\rangle \geq 0.$$
    Adding the above inequalities leads to
    \[
  0 \leq  \langle{F_{\lambda_{l}} (x_{\lambda_{l}}) - F_{\lambda_{l}} (x_{\lambda_{k}})},{x_{\lambda_{k}} - x_{\lambda_{l}}} \rangle+ \langle{F_{\lambda_{l}} (x_{\lambda_{k}}) - F_{\lambda_{k}} (x_{\lambda_{k}})},{x_{\lambda_{k}} - x_{\lambda_{l}}} \rangle. 
    \]
   Rearranging terms and noting that $F_{\lambda_{l}}$ is ${\lambda_{l}} \mu$ strongly monotone, one can obtain
    \begin{align*}
        \lambda_l \mu \|{x_{\lambda_k} - x_{\lambda_l}}\|^2 &\leq (\lambda_l - \lambda_k) \langle{\Phi(x_{\lambda_k})},{x_{\lambda_k} - x_{\lambda_l}} \rangle\leq |\lambda_l - \lambda_k| \|\Phi(x_{\lambda_k})\| \|x_{\lambda_k} - x_{\lambda_l}\|,
    \end{align*}
   where the second inequality is achieved by using the Cauchy-Schwarz inequality. Now, dividing both sides by $\lambda_l \mu\|{x_{\lambda_k} - x_{\lambda_l}}\|$ one can obtain
    \[
    \|{x_{\lambda_{k}} - x_{\lambda_{l}}} \|\leq \frac{\|\Phi (x_{\lambda_{k}})\|}{\mu} \frac{|{\lambda_{l}} - {\lambda_{k}}|}{{\lambda_{l}}}.
    \]
    As $\{x_{\lambda_{k}}\}$ is bounded and $\Phi$ is Lipschitz continuous we know $\Phi( x_{\lambda_{k}})$ is also bounded.
    Let $M=\sup_{k\geq 1}\frac{\|\Phi(x_{\lambda_k})\|}{\mu}$, which is finitie. This yields the desired result 
$    \|x_{\lambda_{k}} - x_{\lambda_{l}}\| \leq M \frac{|{\lambda_{l}} - {\lambda_{k}}|}{{\lambda_{l}}}$. 
\end{proof}

\begin{lemma}\label{bound of inner pro}
    Let $x_{\lambda_k}$ be the solution to the regularized inverse variational inequality $\mathrm{SIVI}(F_{\lambda_k}, X)$. Suppose Assumption \ref{assum:problem} holds. Then for $\bar z_k = \mathbf{P}_{X}\left[ F_{\lambda_k}(x_k) -\eta_{k} x_k\right] $ and $\eta_k>\frac{(L_1+\lambda_kL_2)^2}{2\lambda_k\mu}$ we have:
    \begin{align*}
    \langle  x_k -x_{\lambda_{k}}, \bar z_k- F_{\lambda_k}(x_k)\rangle \leq &-(\lambda_k \mu - \tfrac{(L_1+\lambda_kL_2)^2}{2\eta_k} )\| x_k -x_{\lambda_{k}}\|^2 \\ &-\tfrac{1}{2\eta_k}\|\bar z_k- F_{\lambda_k}(x_k)\|^2.
\end{align*}
\end{lemma}
\begin{proof}
    Since $\bar z_k = \mathbf{P}_{X}\left[ F_{\lambda_k}(x_k) -\eta_{k} x_k\right] $ by Lemma \ref{proj} for all $y \in X$, 
    \begin{align*}
        \langle y- \bar z_k,  F_{\lambda_k}(x_k) -\eta_{k} x_k - \bar z_k\rangle \leq 0 
    \end{align*}
    Taking $y = F_{\lambda_k}(x_{\lambda_k})  \in X$, we deduce that
     \begin{align}\label{proj use}
        \langle \bar z_k-F_{\lambda_k}(x_{\lambda_k}) ,  F_{\lambda_k}(x_k) -\eta_{k} x_k - \bar z_k\rangle \geq 0 
    \end{align}
    On the other hand, since $x_{\lambda_k}$ is a solution of  $\mathrm{SIVI}(F_{\lambda_k}, X)$, we have the following: 
    \begin{align}\label{IVI use def}
        \langle \bar z_k - F_{\lambda_k}(x_{\lambda_k}), \eta_k x_{\lambda_k}\rangle \geq 0  
    \end{align}
Adding \eqref{proj use} and \eqref{IVI use def} leads to 
\begin{align}\label{same 18, 19}
        \langle \bar z_k-F_{\lambda_k}(x_{\lambda_k}) ,  F_{\lambda_k}(x_k) -\eta_{k}( x_k-x_{\lambda_k}) -\bar z_k\rangle \geq 0. 
    \end{align}
By adding and subtracting $F_{\lambda_k}(x_k)$ 
and rearranging terms yields: 
\begin{align*}
       \eta_{k}\langle \bar z_k-F_{\lambda_k}(x_k) ,   x_k-x_{\lambda_k}\rangle &\leq -\eta_{k} \langle F_{\lambda_k}(x_k)-F_{\lambda_k}(x_{\lambda_k}) ,   x_k-x_{\lambda_k} \rangle \\ &\quad -\|\bar z_k-F_{\lambda_k}(x_k)\|^2\\ 
       &\quad +\langle F_{\lambda_k}(x_{k})- \bar z_k, F_{\lambda_k}(x_k)-F_{\lambda_k}(x_{\lambda_k})\rangle
       \\ &\leq  -\eta_k \lambda_k \mu \|x_k-x_{\lambda_k}\|^2 -\tfrac{1}{2}\|\bar z_k - F_{\lambda_k}(x_k)\|^2 \\ &\quad+\tfrac{1}{2}\|F_{\lambda_k}(x_k)-F_{\lambda_k}(x_{\lambda_k}) \|^2
       \end{align*}
where the second inequality follows from strong monotonicity of $F_{\lambda_k}$ and Young's inequality. Using Lipschitz continuity of $F_{\lambda_k}$ and dividing both sides by $\eta_k$, the desired result is achieved. 
\end{proof}

\begin{lemma}\label{exp bound}
     Let $\{x_k\}$ be the sequence generated by Algorithm \ref{alg1}. Suppose Assumptions \ref{assump_error}  and \ref{assum:problem} hold,
     and choose sequences $\lambda_k = (k+1)^{-p}$, $\eta_k = C_1(k+1)^{q}$, and $\theta_k = C_2(k+1)^{-r}$ where $C_1 = \frac{(L_1 + L_2)^2}{\mu}$, $C_2 = \frac{1}{8C_1}$, and $0 < p \leq q \leq r <1$. Then, for all $k\ge 0$, the following pathwise estimates hold:
     \begin{align}\label{equ for main th0}
\nonumber  \| x_{k+1} - x_{\lambda_{k+1}}\|^2 &\leq (1-\tfrac{1}{4}\theta_k\lambda_k\mu)\| x_{k} - x_{\lambda_{k}}\|^2 - \theta_k(\tfrac{1}{2\eta_k}-2\theta_k)\|\bar z_k- F_{\lambda_k}(x_k)\|^2 \\ \nonumber & \quad- 2\theta_k \langle  x_k -x_{\lambda_{k}},\bar w_{k,N_k} \rangle  + \tfrac{M^2|  \lambda_{k} - \lambda_{k+1}|^2}{\lambda_{k}^2} (3+2\theta_k \eta_k  + \tfrac{2}{\lambda_{k} \theta_k \mu})\\ &\quad+\theta_k(10\theta_k+\tfrac{4}{\lambda_k\mu})\|\bar w_{k,N_k}\|^2\\ \nonumber
&\leq \| x_{k} - x_{\lambda_{k}}\|^2 - \theta_k(\tfrac{1}{2\eta_k}-2\theta_k)\|\bar z_k- F_{\lambda_k}(x_k)\|^2 \\ \nonumber & \quad- 2\theta_k \langle  x_k -x_{\lambda_{k}},\bar w_{k,N_k} \rangle  + \tfrac{M^2|  \lambda_{k} - \lambda_{k+1}|^2}{\lambda_{k}^2} (3+2\theta_k \eta_k  + \tfrac{2}{\lambda_{k} \theta_k \mu})\\ &\quad+\theta_k(10\theta_k+\tfrac{4}{\lambda_k\mu})\|\bar w_{k,N_k}\|^2.\label{equ for main th}
\end{align} 
Consequently,
     \begin{align}\label{err bound}
   \nonumber \mathbb{E}[\|x_{k+1} - x_{\lambda_{k+1}}\|^2 \mid \mathcal{F}_k] &\leq \| x_{k} - x_{\lambda_{k}}\|^2 - \theta_k(\tfrac{1}{2\eta_k}-2\theta_k)\|\bar z_k- F_{\lambda_k}(x_k)\|^2 \\ \nonumber&\quad+ \tfrac{M^2|  \lambda_{k} - \lambda_{k+1}|^2}{\lambda_{k}^2} (3+2\theta_k \eta_k  + \tfrac{2}{\lambda_{k} \theta_k \mu}) \\ &\quad+\tfrac{\theta_k \nu^2}{N_k}(10\theta_k+\tfrac{4}{\lambda_k\mu}).
\end{align}
     
\end{lemma}
\begin{proof}
The detailed proof is provided in Appendix A,
where \eqref{equ for main th0} and
\eqref{equ for main th} are derived using Lemma \ref{bound of inner pro}, projection
nonexpansiveness, and Young's inequality. Taking conditional expectation in
\eqref{equ for main th}, using Assumption \ref{assump_error}, gives
\eqref{err bound}.
\end{proof}
The following lemma demonstrates that the distance between our iterates and the regularization path is uniformly bounded.

\begin{lemma}\label{bound x_k}
    Let $\{x_k\}$ be the sequence generated by Algorithm \ref{alg1}. Suppose Assumptions \ref{assump_error}  and \ref{assum:problem} hold. Under the parameter choices $\lambda_k = (k+1)^{-p}$, $\eta_k = C_1(k+1)^{q}$, and $\theta_k = C_2(k+1)^{-r}$ where $C_1 = \frac{(L_1 + L_2)^2}{\mu}$, $C_2 = \frac{1}{8C_1}$, $0 < p \leq q \leq r <1$,  $p +r < 1 $  and $N_k = \lceil(k+1)^{1+\delta}\rceil$ for some $\delta > 0$,  then there exists $D > 0$ such that $\mathbb{E}[\|x_k - x_{\lambda_k}\|^2] \leq D^2$. 
\end{lemma}
\begin{proof}
    Consider \eqref{err bound} in Lemma \ref{exp bound}. Define
    $$A_k\triangleq \big(\tfrac{M^2|  \lambda_{k} - \lambda_{k+1}|^2}{\lambda_{k}^2} (3+2\theta_k \eta_k  + \tfrac{2}{\lambda_{k} \theta_k \mu})+\tfrac{\theta_k \nu^2}{N_k}(10\theta_k+\tfrac{4}{\lambda_k\mu})\big)$$
    Substituting the sequences $\lambda_k = (k+1)^{-p}$, $\eta_k = C_1(k+1)^{q}$, and $\theta_k = C_2(k+1)^{-r}$, and applying Lemma \ref{MVT} with $h(k) = (k+1)^{-p}$ yield  
    \begin{align*}
        \sum_{k=0}^{\infty} A_k &\leq  \underbrace{\sum_{k=0}^{\infty} (Mp)^2(k+1)^{-2} \left(3 + \tfrac{(k+1)^{q-r}}{4} + \tfrac{2(k+1)^{p+r}}{\mu C_2}\right)}_{\text{term (a)}} \\ &\quad + \underbrace{\sum_{k=0}^{\infty} \tfrac{C_2 \nu^2 (k+1)^{-r}}{N_k} \left(10C_2(k+1)^{-r} +\tfrac{4(k+1)^{p}}{\mu}  \right)}_{\text{term (b)}} .
    \end{align*} 

   Since $q-r \leq 0$ and $p+r < 1$, then $-2+p+r < -1  $ and $-2+q-r < -2$, we conclude that term (a) is summable. Furthermore, consider $N_k = \lceil(k+1)^{1+\delta}\rceil$ for some $\delta >0$, since $0 < p \leq r$ it is easy to show that $-2r-1-\delta < -1$ and $-1+p-r-\delta< -1$ which leads term (b) also is summable, i.e., $\sum_{k=0}^{\infty} A_k < \infty.$
   Taking total expectation in Lemma \ref{exp bound} and dropping the nonpositive residual term yields $\mathbb E[\|x_{k+1}-x_{\lambda_{k+1}}\|^2]\leq \mathbb E[\|x_{k}-x_{\lambda_{k}}\|^2]+A_k$. Recursion gives:
   $$\mathbb E[\|x_k-x_{\lambda_k}\|^2]
\le
\mathbb E[\|x_0-x_{\lambda_0}\|^2]+\sum_{j=0}^{k-1}A_j
\le
\mathbb E[\|x_0-x_{\lambda_0}\|^2]+\sum_{j=0}^{\infty}A_j\triangleq D^2<\infty.$$ 
\end{proof} 
We first examine the limiting behavior of the regularized solutions
$\{x_{\lambda_k}\}$. Since the algorithm uses $\lambda_k=(k+1)^{-p}$ with
$p>0$, we have $\lambda_k\downarrow 0$. The following lemma shows that the
regularization path approaches the solution set of the original SIVI.
\begin{lemma}[Limit points of the regularization path]\label{lem:reg_path_limit}
Suppose Assumption~\ref{assum:problem} holds and let $\lambda_k\downarrow 0$. Then the sequence
$\{x_{\lambda_k}\}$ is bounded. Moreover, every limit point of
$\{x_{\lambda_k}\}$ belongs to the solution set $S$ of the original SIVI. Consequently, $\operatorname{dist}(x_{\lambda_k},S)\to 0.$
\end{lemma}

\begin{proof}
The boundedness of $\{x_{\lambda_k}\}$ follows from Lemma~\ref{M bound}. Let $\bar x$ be an
arbitrary limit point of $\{x_{\lambda_k}\}$. Then there exists a subsequence,
still denoted by $\{x_{\lambda_{k_j}}\}$, such that $x_{\lambda_{k_j}}\to \bar x$.
Since $x_{\lambda_{k_j}}$ solves the regularized SIVI associated with
$F_{\lambda_{k_j}}=F+\lambda_{k_j}\Phi$, we have
$F_{\lambda_{k_j}}(x_{\lambda_{k_j}})\in X$ and
\[
    \langle y-F_{\lambda_{k_j}}(x_{\lambda_{k_j}}),
    x_{\lambda_{k_j}}\rangle\geq 0,\qquad \forall y\in X.
\]
Since $\{x_{\lambda_{k_j}}\}$ is bounded and $\Phi$ is Lipschitz continuous,
the sequence $\{\Phi(x_{\lambda_{k_j}})\}$ is bounded. Hence
$\lambda_{k_j}\Phi(x_{\lambda_{k_j}})\to 0$. By continuity of $F$, we obtain
\[
    F_{\lambda_{k_j}}(x_{\lambda_{k_j}})
    =
    F(x_{\lambda_{k_j}})
    +
    \lambda_{k_j}\Phi(x_{\lambda_{k_j}})
    \to F(\bar x).
\]
Because $X$ is closed and
$F_{\lambda_{k_j}}(x_{\lambda_{k_j}})\in X$, it follows that $F(\bar x)\in X$.
Passing to the limit in the inequality above gives
\[
    \langle y-F(\bar x),\bar x\rangle\geq 0,\qquad \forall y\in X.
\]
Therefore, $\bar x\in S$.
It remains to show that $\operatorname{dist}(x_{\lambda_k},S)\to 0$. Suppose not.
Then there exist $\alpha>0$ and a subsequence $\{x_{\lambda_{k_j}}\}$ such that
$\operatorname{dist}(x_{\lambda_{k_j}},S)\geq \alpha$ for all $j$. Since
$\{x_{\lambda_{k_j}}\}$ is bounded, it has a further convergent subsequence whose
limit belongs to $S$ by the first part of the proof. This contradicts
$\operatorname{dist}(x_{\lambda_{k_j}},S)\geq \alpha$. Hence
$\operatorname{dist}(x_{\lambda_k},S)\to 0$.
The proof is complete. 
\end{proof}
We now combine the one-step recursion from Lemma~\ref{exp bound} with the
Robbins--Siegmund theorem to show that the stochastic iterates track the
regularization path almost surely. As a consequence, their distance to the
solution set of the original SIVI converges to zero almost surely.
\begin{theorem}[Almost sure convergence to the solution set]\label{thm:as_convergence_set}
Let $\{x_k\}$ be the sequence generated by Algorithm~\ref{alg1}. Suppose Assumptions~\ref{assump_error}
and~\ref{assum:problem} hold. Let
$    \lambda_k=(k+1)^{-p},$ $\eta_k=C_1(k+1)^q,$ $\theta_k=C_2(k+1)^{-r},$
where $C_1=(L_1+L_2)^2/\mu$, $C_2=1/(8C_1)$, and
$0<p\leq q\leq r<1$ with $p+r<1$. Let
$N_k=\lceil(k+1)^{1+\delta}\rceil$ for some $\delta>0$. Then
\[
    \|x_k-x_{\lambda_k}\|\to 0
    \qquad \text{almost surely}.
\]
Consequently, $\{x_k\}$ is almost surely bounded and
\[
    \operatorname{dist}(x_k,S)\to 0
    \qquad \text{almost surely}.
\]
In particular, every limit point of $\{x_k\}$ belongs to the solution set $S$ almost surely.
\end{theorem}

\begin{proof}
Define $V_k:=\|x_k-x_{\lambda_k}\|^2.$
Taking conditional expectation in \eqref{equ for main th0} and  using Assumption \ref{assump_error}, we get
\[
\mathbb{E}[V_{k+1}\mid \mathcal{F}_k]
\leq
\left(1-\frac{\mu}{4}\theta_k\lambda_k\right)V_k + A_k,
\]
where
\[
A_k :=
\frac{M^2|\lambda_k-\lambda_{k+1}|^2}{\lambda_k^2}
\left(3+2\theta_k\eta_k+\frac{2}{\lambda_k\theta_k\mu}\right)
+
\frac{\theta_k\nu^2}{N_k}
\left(10\theta_k+\frac{4}{\lambda_k\mu}\right).
\]
The summability argument in Lemma~\ref{bound x_k} shows that
$\sum_{k=0}^{\infty} A_k < \infty.$
Moreover, $\sum_{k=0}^{\infty} \theta_k\lambda_k
    =
    C_2\sum_{k=0}^{\infty}(k+1)^{-(p+r)}
    =
    \infty,$
because $p+r<1$. By Robbins--Siegmund (Lemma~\ref{lem:RS}), $V_k:=\|x_k-x_{\lambda_k}\|^2$ converges almost surely and $\sum_{k=0}^\infty \theta_k\lambda_k V_k<\infty$ almost surely. Since $\sum_{k=0}^\infty\theta_k\lambda_k=C_2\sum_{k=0}^\infty(k+1)^{-(p+r)}=\infty$ because $p+r<1$, the almost sure limit of $V_k$ must be zero. Hence $\|x_k-x_{\lambda_k}\|\to 0$ almost surely. Since $\{x_{\lambda_k}\}$ is bounded by Lemma~\ref{M bound}, the relation
$\|x_k-x_{\lambda_k}\|\to 0$ implies that $\{x_k\}$ is almost surely bounded.
Furthermore, by Lemma~\ref{lem:reg_path_limit},
$\operatorname{dist}(x_{\lambda_k},S)\to 0$. Therefore,
\[
    \operatorname{dist}(x_k,S)
    \leq
    \|x_k-x_{\lambda_k}\|
    +
    \operatorname{dist}(x_{\lambda_k},S)
    \to 0
    \qquad \text{almost surely}.
\]
Finally, since $S$ is closed (see Lemma \ref{closedSol}), every limit point of $\{x_k\}$ belongs to $S$ almost surely. 
\end{proof}

In the following theorem, we provide the main result of this paper by finding a bound for the 
residual map.
\begin{theorem}\label{main th}
    Let $\{x_k\}_{k\geq 0}$ be a sequence generated by Algorithm \ref{alg1}. Suppose Assumptions \ref{assump_error} and \ref{assum:problem} hold. For the parameter sequences $\lambda_k = (k+1)^{-p}$, $\eta_k = C_1(k+1)^{q}$, and $\theta_k = C_2(k+1)^{-r}$ where $C_1 = \frac{(L_1 + L_2)^2}{\mu}$, $C_2 = \frac{1}{8C_1}$, and $0 < p \leq q \leq r <1$,   $p +r < 1 $ and assuming a batch size $N_k = \lceil(k+1)^{1+\delta}\rceil$ for some $\delta > 0$. Then there exist constants $C,D>0$, independent of $T$, such that
  \begin{align}\label{upper bound for exp}
\nonumber &\min_{0\leq k\leq T-1} \mathbb E[\|H(x_k,\eta_k)\|^2] \\ \nonumber&\leq \frac{1}{T}\Big[(8 p C_1 M)^2 \sum_{k=0}^{T-1} (k+1)^{-2+r+q} \left( \tfrac{16 C_1 (k+1)^{p+r}}{\mu} + \tfrac{(k+1)^{q-r}}{4} + 3 \right) \\ \nonumber&\quad
 + 8C_1\nu^2  \sum_{k=0}^{T-1} {(k+1)^{r+q-1-\delta}} \left( 10 C_2 (k+1)^{-2r} + \tfrac{4 (k+1)^{-r+p}}{\mu} \right) \\
&\quad + \sum_{k=0}^{T-1} 8(k+1)^{-2p} C + 64 C_1^2 D^2 T^{q+r}\Big].
\end{align}
\end{theorem}
\begin{proof}
The detailed derivation is provided in Appendix B. There,
we rearrange the one-step estimate \eqref{equ for main th}, use the residual comparison
between $H(x_k,\eta_k)$ and $H_{\lambda_k}(x_k,\eta_k)$, sum over
$k=0,\ldots,T-1$, and apply Assumption \ref{assump_error} and Lemma \ref{bound x_k}. This gives the
cumulative residual bound, and dividing by $T$ yields the stated result.
\end{proof}

\begin{corollary} \label{stochasticCor}
    Under the premises of Theorem \ref{main th}, choosing $p = q = r = 1/4$, the following holds:
 \begin{enumerate}\item[(i)] For any $T \geq 1$: $ \min_{0\leq k\leq T-1} \mathbb{E}\big[\|H(x_k,\eta_k)\|^2\big]
= \mathcal{O}\left(\tfrac{1}{\sqrt{T}}\right).$
\item[(ii)] To compute an $\epsilon$-solution, i.e., $\min_{k} \mathbb E[\|H(x_k, \eta_k)\|^2] \leq \epsilon$, the total number of stochastic oracle calls is $\mathcal{O}(1/\epsilon^{4+2\delta})$.\end{enumerate}
\end{corollary}
\begin{proof}

(i) Substituting $p = q = r = 1/4$ into \eqref{upper bound for exp}, we obtain
\begin{align}\label{bound before int}
\nonumber \min_k \mathbb E[\|H(x_k,\eta_k)\|^2] &\leq \frac{1}{T}\Big[(2 C_1 M)^2 \sum_{k=0}^{T-1}  \left( \tfrac{16 C_1 }{\mu (k+1)} + \tfrac{13 (k+1)^{-1.5}}{4} \right) \\&\quad \nonumber + 8C_1\nu^2  \sum_{k=0}^{T-1}  \left( \tfrac{10 C_2 }{(k+1)^{1+\delta}} + \tfrac{4 }{\mu(k+1)^{0.5+\delta}} \right)\\&\quad
  + \sum_{k=0}^{T-1} 8(k+1)^{-0.5} C + 64 C_1^2 D^2 T^{0.5}\Big].
\end{align}
Now bounding each summation using the integral bound $\sum_{k=0}^{T-1} (k+1)^{-m} \leq 1 + \int_{1}^{T} x^{-m} dx $, one can obtain the following
inequalities using simple algebra.
\begin{align*}
&\sum_{k=0}^{T-1} (k+1)^{-1} \leq 1 + \ln T\\
&\sum_{k=0}^{T-1} (k+1)^{-1.5} \leq 1 + \int_{1}^{T} x^{-1.5} dx = 1 + [-2x^{-0.5}]_1^T = 3 - \frac{2}{\sqrt{T}} \leq 3\\
&\sum_{k=0}^{T-1}\tfrac{1}{(k+1)^{1+\delta}} \leq 1+ \sum_{k=1}^{T-1}\tfrac{1}{(k+1)^{1+\delta}} \leq 1+\int_{0}^{T-1} \tfrac{1}{(x+1)^{1+\delta}} \,dx \leq 1+\tfrac{1}{\delta}-\tfrac{1}{T\delta} \leq 1+\tfrac{1}{\delta}\\
&\sum_{k=0}^{T-1} (k+1)^{-(0.5+\delta)} \leq \sum_{k=0}^{T-1} (k+1)^{-0.5} \leq 1 + [2x^{0.5}]_1^T = 2\sqrt{T} - 1\\
\end{align*}
Next, using the above inequalities in \eqref{bound before int} yields 

\begin{align*}
\nonumber \min_k \mathbb E[\|H(x_k,\eta_k)\|^2] &\leq \frac{1}{T}\Big[(2 C_1 M)^2   \left( \tfrac{16 C_1 (1+\ln T)}{\mu} + \tfrac{39}{4} \right) \\&\quad+ 8C_1\nu^2  \left( {10 C_2 (1+\tfrac{1}{\delta})}+ \tfrac{4 }{\mu}(2\sqrt{T} - 1) \right) \\&\quad
  +  8C(2\sqrt{T}-1)  + 64 C_1^2 D^2 T^{0.5}\Big].
\end{align*}
Rearranging terms and discarding the negative components, we get:
\begin{align*}
\nonumber \min_k \mathbb E[\|H(x_k,\eta_k)\|^2] &\leq \tfrac{(8 M)^2 C_1^3 }{\mu}   \left( \tfrac{\ln T}{ T}  \right) +   \tfrac{64 C_1\nu^2   }{\mu \sqrt{T}}
  +  \tfrac{16C}{\sqrt{T}}  + \tfrac{64 C_1^2 D^2 }{\sqrt{T}}+\tfrac{\Delta}{T} = \mathcal{O}(\tfrac{1}{\sqrt{T}}).
\end{align*}
where $\Delta = (2 C_1 M)^2(\tfrac{ 16 C_1 }{\mu}+ \tfrac{39 }{4})+ 10\nu^2  {  (1+\tfrac{1}{\delta})}$ for some $\delta > 0$. \\

(ii) Since at each iteration $k$ of the method, $N_k = \lceil(k+1)^{1+\delta}\rceil$ samples are required, then we have $$   \sum_{k=0}^{T-1}\lceil(k+1)^{1+\delta}\rceil
   \le T+\sum_{k=0}^{T-1}(k+1)^{1+\delta}
   =O(T^{2+\delta}) = \mathcal{O}\left(\frac{1}{\epsilon^{4+2\delta}}\right).$$
\end{proof}

The preceding corollary establishes a rate for the residual $H(x_k,\eta_k)$ evaluated with the iteration-dependent parameter $\eta_k$. Although $H(x,\eta)=0$ characterizes a solution for every $\eta>0$, the residual mapping changes with $k$, and its magnitude is affected by the increasing sequence $\{\eta_k\}$. For a consistent optimality measure that can be compared across iterations, it is therefore useful to show that the same convergence rate also holds for a fixed parameter $\bar\eta>0$. The following corollary establishes this result.

\begin{corollary}
    For the iterates generated under the conditions of Corollary \ref{stochasticCor}, for any fixed $\bar\eta\in(0,C_1]$,
    $$\min_{0\leq k\leq T-1}\mathbb E[\|H(x_k,\bar\eta)\|^2]=\mathcal O(1/\sqrt T).$$
\end{corollary}
\begin{proof}
Let $0<\eta_1\leq \eta_2$, and define
$p_i:=P_X\bigl(F(x)-\eta_i x\bigr),$  for $i=1,2$. By the projection optimality condition,
\begin{align*}
\left\langle F(x)-\eta_1x-p_1,\;p_2-p_1\right\rangle \leq 0, \quad
\left\langle F(x)-\eta_2x-p_2,\;p_1-p_2\right\rangle \leq 0.
\end{align*}
Summing these two inequalities, we obtain
$\|p_2-p_1\|^2
\leq
-(\eta_2-\eta_1)\left\langle x,\;p_2-p_1\right\rangle.$
Hence,
$\left\langle x,\;p_2-p_1\right\rangle \leq 0.$
Moreover, from the first projection inequality,
$\left\langle F(x)-p_1,\;p_2-p_1\right\rangle
\leq
\eta_1\left\langle x,\;p_2-p_1\right\rangle 
\leq 0.$
Therefore,
\begin{align*}
\|H(x,\eta_2)\|^2
&=
\|F(x)-p_2\|^2 \\
&=
\|F(x)-p_1-(p_2-p_1)\|^2 \\
&=
\|F(x)-p_1\|^2
-2\left\langle F(x)-p_1,\;p_2-p_1\right\rangle
+\|p_2-p_1\|^2 \\
&\geq
\|F(x)-p_1\|^2 
=
\|H(x,\eta_1)\|^2.
\end{align*}
Thus, $\|H(x,\eta)\|$ is nondecreasing in $\eta$.
Consequently, since $\eta_k=C_1(k+1)^{1/4}\ge C_1\ge \bar\eta$, we obtain $\|H(x_k,\bar\eta)\|\le \|H(x_k,\eta_k)\|$.
Hence,
\begin{align*}
\min_{0\leq k\leq T-1}
\mathbb{E}\!\left[\|H(x_k,\bar{\eta})\|^2\right]
&\leq
\min_{0\leq k\leq T-1}
\mathbb{E}\!\left[\|H(x_k,\eta_k)\|^2\right] 
=
\mathcal{O}\!\left(1/\sqrt T\right).
\end{align*}
\end{proof}
\subsection{Deterministic IVI} 
We now consider a deterministic IVI problem. Let $X \subseteq \mathbb{R}^n$ be a nonempty closed convex set, and $F:\mathbb{R}^n \to \mathbb{R}^n$ be a continuous nonlinear map. 
We seek an $x^* \in \mathbb{R}^n$ such that $F(x^*) \in X$ and
\begin{align}\label{determIVI}\tag{IVI}
    \langle y - F(x^*),x^* \rangle \geq 0, \quad \forall y \in X
\end{align}

\begin{algorithm}[H]\normalsize
\caption{Regularized Inverse Projected Gradient (R-IPG) method}
\label{alg2}
\begin{algorithmic}
\STATE {\bf Input}: $x_0 \in \mathbb R^n$, $\{\eta_{k}, \theta_k, \lambda_k\}_{k\geq 0} \subseteq \mathbb{R}_+$;
\FOR{$k=0, \dots, T-1$}

\STATE  $z_k = \mathbf{P}_{X}\left[ F(x_k)  + \lambda_k \Phi(x_k) - \eta_{k} x_k\right]$
\STATE  $x_{k+1} = x_k - \theta_k \left( F(x_k)  + \lambda_k \Phi(x_k) - z_k \right)$
\ENDFOR
\end{algorithmic}
\end{algorithm}

As mentioned earlier, monotone deterministic IVIs have been studied in \cite{anh2024regularized,luo2014tikhonov}, where convergence of regularized methods was established under mild monotonicity assumptions. However, to the best of our knowledge, no explicit convergence-rate result has been obtained for the monotone deterministic setting. In this section, we adopt a regularization strategy similar to the one used for the stochastic problem \eqref{sIvI} by introducing a regularizer \(\Phi:\mathbb{R}^n\to\mathbb{R}^n\) that is \(\mu\)-strongly monotone and Lipschitz continuous. This approach enables us to establish an explicit convergence rate for the deterministic inverse variational inequality problem.
We define for $\lambda > 0$ the regularized operator
\[F_\lambda(x) \triangleq F(x) + \lambda \Phi(x)\]
Given the deterministic nature of the problem, we modify the (RVC-IPG) algorithm by removing the increasing batch-size scheme, yielding a Regularized Inverse Projected Gradient method (R-IPG) which is outlined in Algorithm \ref{alg2}.
The new algorithm is recovered from Algorithm \ref{alg1} by setting the batch size $N_k = 1$ and replacing the sampling with a direct call to $F$.

To measure how far the iterates are from the optimal solution, we examine the degree to which the optimality condition \eqref{opt-cond} is violated using the same residual map $H(x,\eta)$ defined in Section \ref{proposed}.
In the following theorem, we demonstrate the iteration and oracle complexity for the deterministic case. The proof follows the same argument as Theorem~\ref{main th}, with exact evaluations of the mean operator $F$. Equivalently, the stochastic error satisfies $\bar w_{k,N_k}=0$, and all batch-variance terms vanish.
\begin{theorem}[Deterministic IVI]\label{thm:deterministic} Suppose Assumption~\ref{assum:problem} holds, and let the parameter sequences be chosen as in Theorem~\ref{main th} for the deterministic setting. Then the following statements hold:
\begin{enumerate}\item[(i)] For any $T \geq 1$: $\min_{k } \|H(x_k, \eta_k)\|^2 = \mathcal{O}\left(1/{\sqrt{T}}\right)$
\item[(ii)] To compute an $\epsilon$-solution, i.e., $\min_{k} \|H(x_k, \eta_k)\|^2 \leq \epsilon$, the total number of operator calls is $\mathcal{O}(1/\epsilon^2)$.\end{enumerate}\end{theorem}
\begin{proof} The result follows from the proof of Theorem~\ref{main th} by considering exact evaluations of the mean operator $F$. Equivalently, we set $\bar w_{k,N_k}=0$ for all $k$, so all stochastic-error and conditional-variance terms vanish.
        Thus we restate \eqref{equ for main th} from Lemma \ref{exp bound} as 
  \begin{align*}\label{eq:det_recursion}
    \|x_{k+1} - x_{\lambda_{k+1}}\|^2
    \leq &
    \|x_k - x_{\lambda_k}\|^2
    - \theta_k \left(\tfrac{1}{2\eta_k} - 2\theta_k\right)
      \|\bar z_k - F_{\lambda_k}(x_k)\|^2\\
    &+ \frac{M^2|\lambda_k - \lambda_{k+1}|^2}{\lambda_k^2}
      \left(3 + 2\theta_k\eta_k + \frac{2}{\lambda_k\theta_k\mu}\right),
  \end{align*}
  where $\bar z_k = \mathbf{P}_X[F_{\lambda_k}(x_k) - \eta_k x_k]$. 
  The same boundedness argument as in Lemma \ref{bound x_k} applies after removing the stochastic-error terms.
  In the same manner as Theorem \ref{main th} we obtain the following bound
    \begin{align*}
        &\sum_{k = 0}^{T-1} \|H(x_k,\eta_k)\|^2\\ &\leq      4\Gamma_{T-1}D^2
        \\ &+ (8pC_1 M)^2 \sum_{k=0}^{T-1}
            (k+1)^{-2+r+q}
            \left(\frac{16C_1(k+1)^{p+r}}{\mu}
                + \frac{(k+1)^{q-r}}{4} + 3\right) \\
        &+ \sum_{k=0}^{T-1} 8(k+1)^{-2p}C
    \end{align*}
    where $C >0$ and bounds $\sup_k \| \Phi(x_k) \|^2$ and $\Gamma_{T-1} = \mathcal{O}(T^{q+r})$.
    Choosing $p = q = r = \tfrac{1}{4}$ and applying the same integral bounds as in Corollary \ref{stochasticCor} gives
    \begin{align}
        \min_{0\leq k\leq T-1}\|H(x_k,\eta_k)\|^2 \leq \frac{1}{T} \left[ C_1^3\,M^2\,\frac{\ln T}{\mu}  + \sqrt{T}\right] = \mathcal{O}(\frac{1}{\sqrt{T}})
    \end{align} 
Moreover, since the deterministic method uses one exact evaluation of $F$ per iteration, an $\epsilon$-solution requires $O(\epsilon^{-2})$ operator evaluations. 
  
\end{proof}

\section{Numerical Experiments}\label{Numeric}
In this section, to illustrate the effectiveness of the proposed method, we solve a network equilibrium problem with asymmetric interactions and investigate the convergence behavior of the proposed method under monotone mappings. All experiments are performed in \zi{Python 3.14.2} on a machine running 64-bit Windows 11 with Intel i5-1135G7 @2.40GHz and 8GB RAM.

{\bf Network Equilibrium with Asymmetric Interactions.}
In this example, we evaluate the proposed RVC-IPG method on a synthetic network equilibrium problem with asymmetric interactions. The purpose of this experiment is to examine the performance of the proposed method in a setting that is monotone and Lipschitz continuous but not co-coercive. Such a setting is not covered by existing projection-based IVI methods that rely on co-coercivity.
Let $n=100$, and define the mean operator $F:\mathbb{R}^n\to\mathbb{R}^n$ by $F(x)=(M+S)x\zi{+b}$, where $M\in\mathbb{R}^{n\times n}$ is a diagonal positive semidefinite matrix, $S\in\mathbb{R}^{n\times n}$ is skew-symmetric, \zi{and $b\in\mathbb{R}^n$ is a
fixed offset vector}. \zi{The variable $x$ represents the deviation of the network state or control vector from a prescribed nominal operating point.} Specifically, we set $M=\operatorname{diag}(0,0,0.5,\ldots,0.5)$ and construct $S$ as a block-diagonal matrix with $2\times 2$ skew-symmetric blocks,
\begin{align*}
S=\operatorname{blkdiag}\left(
\begin{bmatrix}
0 & -\rho_1\\
\rho_1 & 0
\end{bmatrix},
\ldots,
\begin{bmatrix}
0 & -\rho_{n/2}\\
\rho_{n/2} & 0
\end{bmatrix}
\right),
\end{align*}
where each $\rho_i$ is sampled independently from the uniform distribution on $\zi{[1,2]}$. The matrix $M$ models direct monotone effects, while $S$ models asymmetric interactions among network components. \zi{The offset $b$ is generated as
$b=-(M+S)\bar x$ with  $\bar{x}=\tfrac14\mathbf{1}$, } and the feasible set is chosen as $X=[0,2]^n$.
The stochastic oracle is given by $G(x,\xi)=F(x)+\xi$, where $\xi\sim\mathcal{N}(0,\sigma^2 I_n)$. Hence, $\mathbb{E}[G(x,\xi)]=F(x)$. For the additive Gaussian oracle, the batch average is sampled equivalently as $F(x_k)+\bar\xi_k$, where $\bar\xi_k\sim\mathcal N(0,\sigma^2N_k^{-1}I_n)$. We consider $\sigma\in \zi{\{ 0.5, 5, 50\}}$. The regularization operator is chosen as $\Phi(x)=x$, which is $1$-strongly monotone and $1$-Lipschitz continuous.

We next verify that this example satisfies the assumptions of our analysis while falling outside the co-coercive setting. For any $x,y\in\mathbb{R}^n$, let $z=x-y$. Since $M\succeq 0$ and $S^\top=-S$, we have $\langle F(x)-F(y),x-y\rangle=\langle (M+S)z,z\rangle=\langle Mz,z\rangle+\langle Sz,z\rangle=\langle Mz,z\rangle\geq 0$. Thus, $F$ is monotone. Moreover, since $F$ is linear, it is Lipschitz continuous with constant $L_1=\|M+S\|_2$. However, $F$ is not co-coercive. Indeed, because the first two diagonal entries of $M$ are zero, the first $2\times 2$ block contains a purely skew-symmetric component. Taking any nonzero vector $z$ supported only on the first two coordinates gives $\langle (M+S)z,z\rangle=0$, while $\|(M+S)z\|^2>0$. Therefore, there exists no constant $\beta>0$ such that $\langle F(x)-F(y),x-y\rangle\geq \beta\|F(x)-F(y)\|^2$ for all $x,y\in\mathbb{R}^n$.

The SIVI problem is to find $x^*\in\mathbb{R}^n$ such that $F(x^*)\in X$ and $\langle y-F(x^*),x^*\rangle\geq 0$ for all $y\in X$. 
We implement Algorithm~\ref{alg1} with the parameter choices $\lambda_k=(k+1)^{-1/4}$, $\eta_k=C_1(k+1)^{1/4}$, and $\theta_k=C_2(k+1)^{-1/4}$, where $C_1=(L_1+L_2)^2/\mu$, $C_2=1/(8C_1)$, and $L_2=\mu=1$. We compare the theoretically motivated increasing batch sizes $N_k=\lceil (k+1)^{1.1}\rceil$ and $N_k=\lceil (k+1)^{2.1}\rceil$ with the constant-batch baseline $N_k=100$. The algorithm is run for \zi{$10^5$} iterations. To estimate the expected residual map, we average over $50$ independent simulation runs initialized from the same randomly generated point $x_0\in X$. The confidence intervals are computed using the Student $t$-distribution. 
\begin{figure}[htb]
    \centering\vspace{-0.5cm}
\subfloat[Effect of $\sigma$ with fixed $N_k=\lceil (k+1)^{1.1}\rceil$]{\includegraphics[scale=0.24]{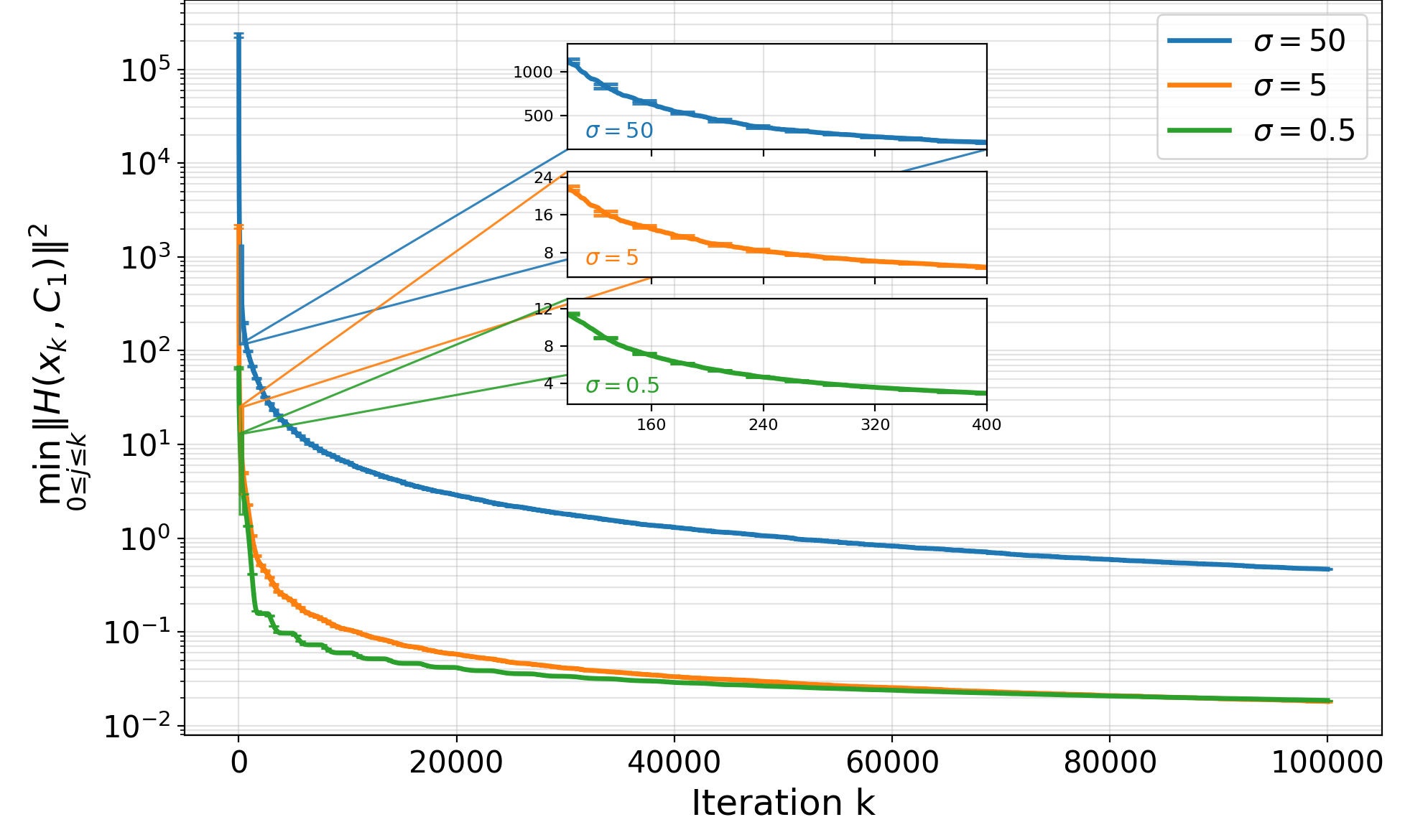}}
\subfloat[Effect of $N_k$ with fixed $\sigma=50$]{\includegraphics[scale=0.24]{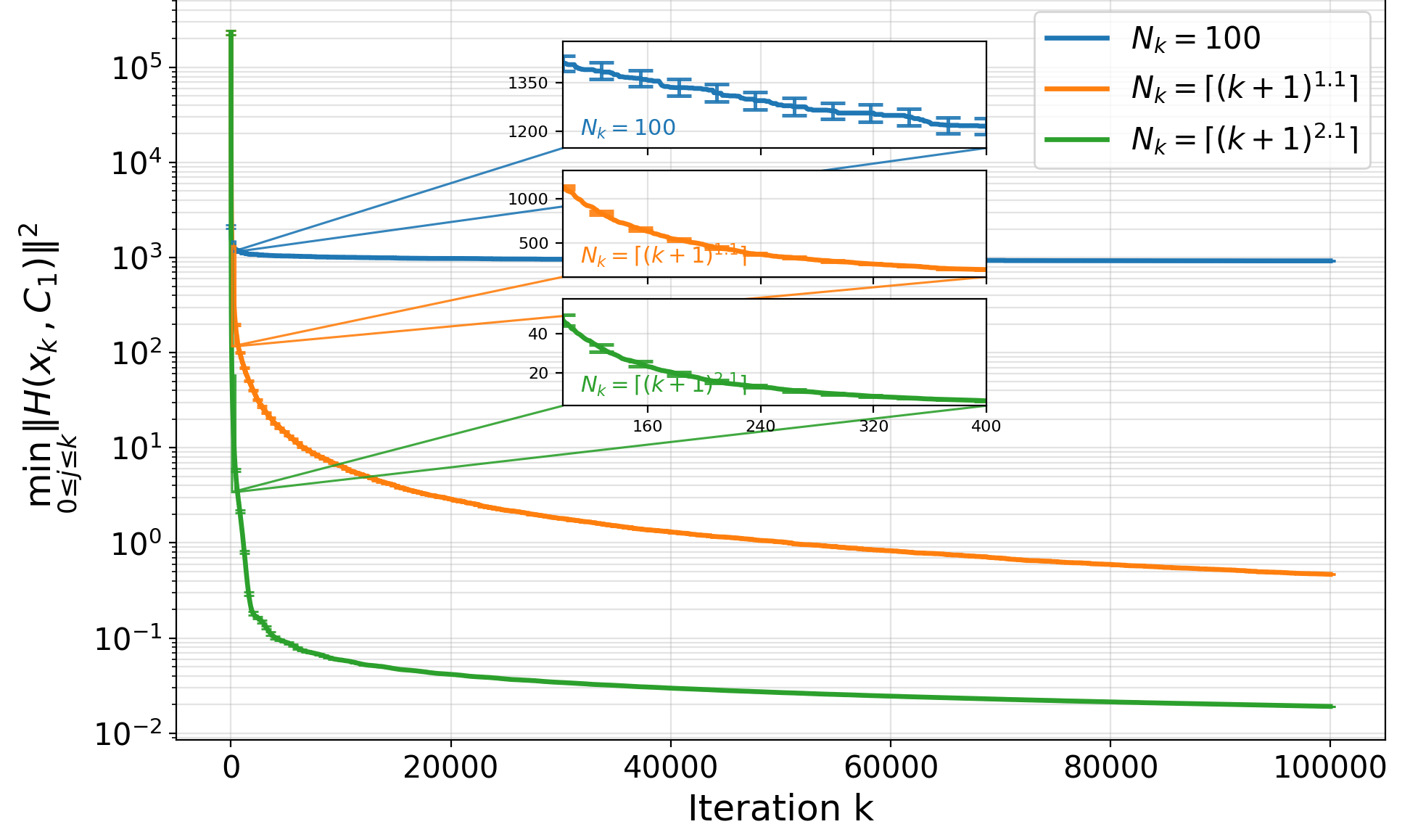}}
     \caption{\zi{Pointwise residual $\mathbb{E}[\|H(x_k,\bar\eta)\|^2]$ with
  $\bar\eta=C_1$ }. The error bars represent the $95\%$ confidence intervals over $50$ independent runs.}
    \label{fig:MS_pointwise_gap}
\end{figure}

\begin{figure}[htb]
    \centering\vspace{-0.5cm}
\subfloat[Effect of $\sigma$ with fixed $N_k=\lceil (k+1)^{1.1}\rceil$]{\includegraphics[scale=0.24]{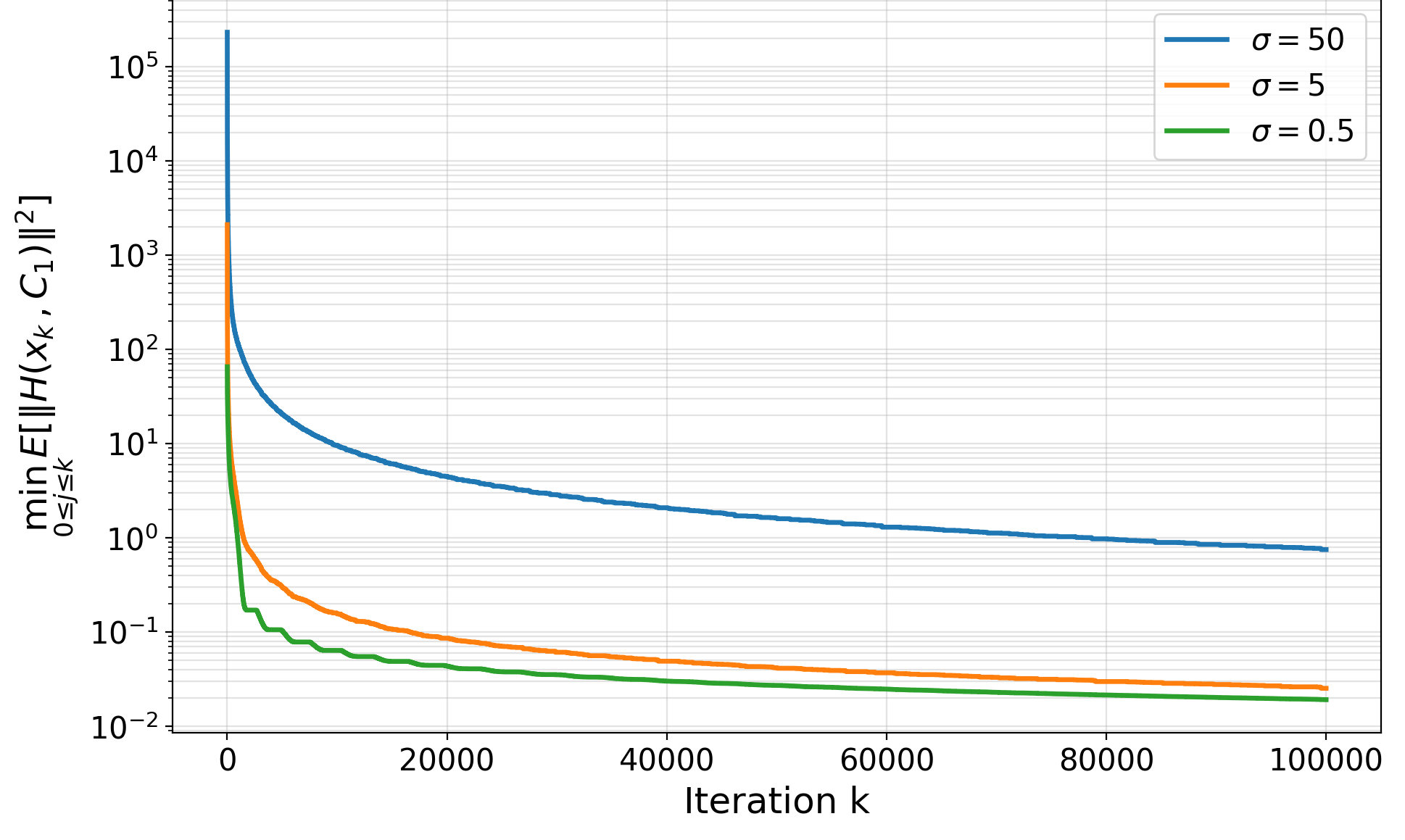}}
\subfloat[Effect of $N_k$ with fixed $\sigma=50$]{\includegraphics[scale=0.24]{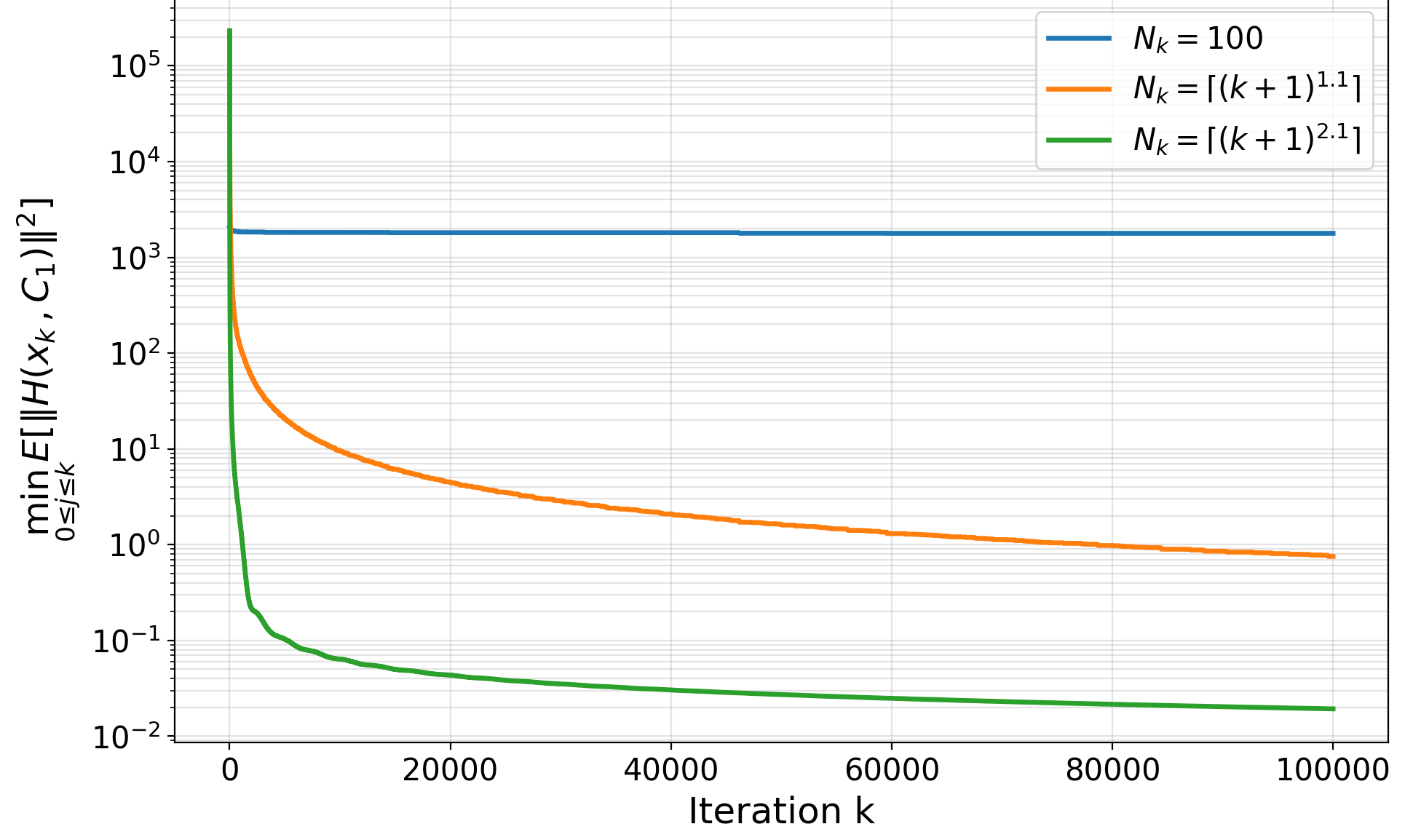}}
      \caption{Best residual up to iteration $k$, given by the sample average of
$\min_{0\le j\le k}\|H(x_j,\bar\eta)\|^2$ with $\bar\eta=C_1$.}
    \label{fig:MS_best_gap}
\end{figure}

\begin{figure}[htb]
    \centering\vspace{-0.5cm}
     \subfloat[Effect of $\sigma$ with fixed $N_k=\lceil (k+1)^{1.1}\rceil$]{\includegraphics[scale=0.24]{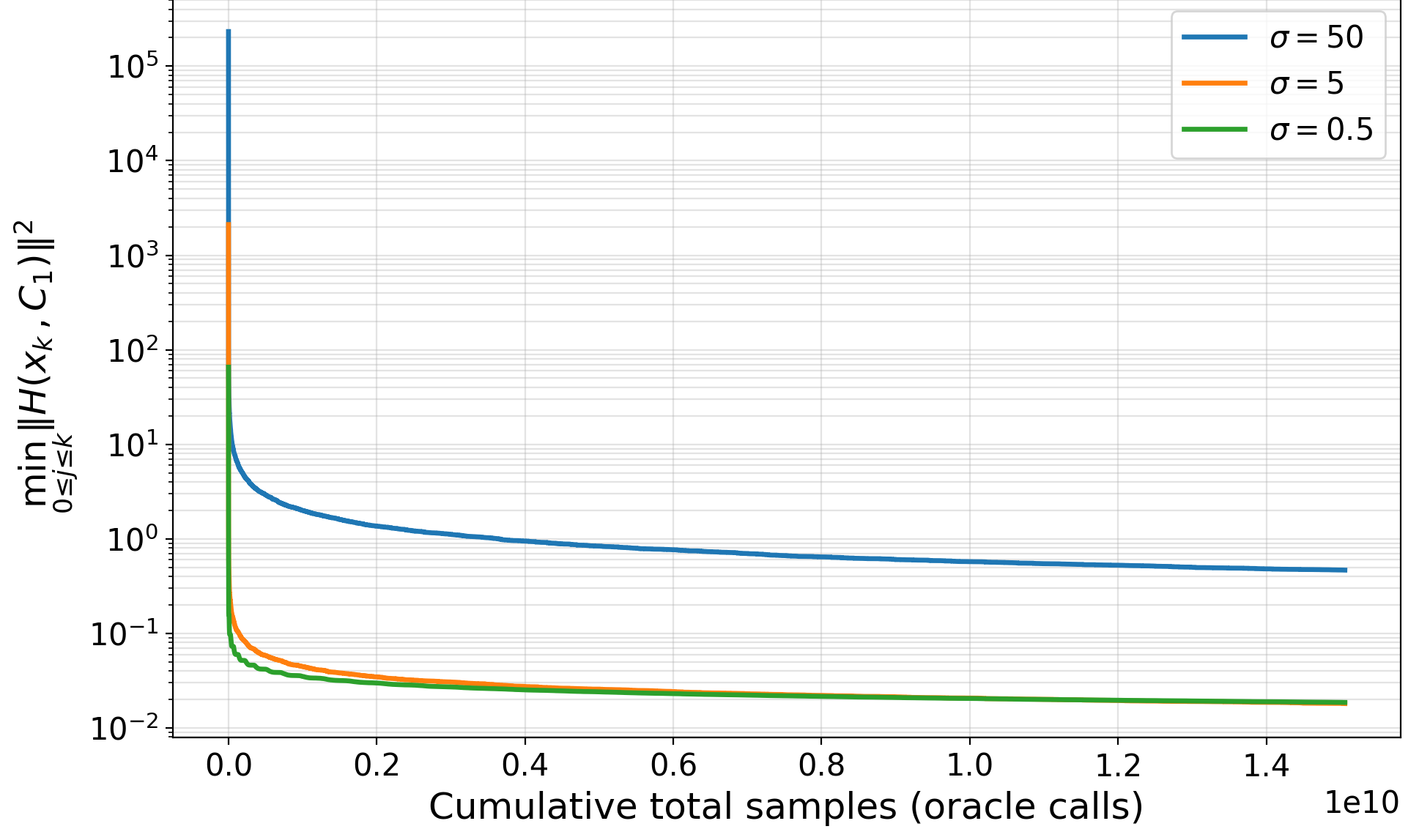}}
    \subfloat[Effect of $N_k$ with fixed $\sigma=50$]{\includegraphics[scale=0.24]{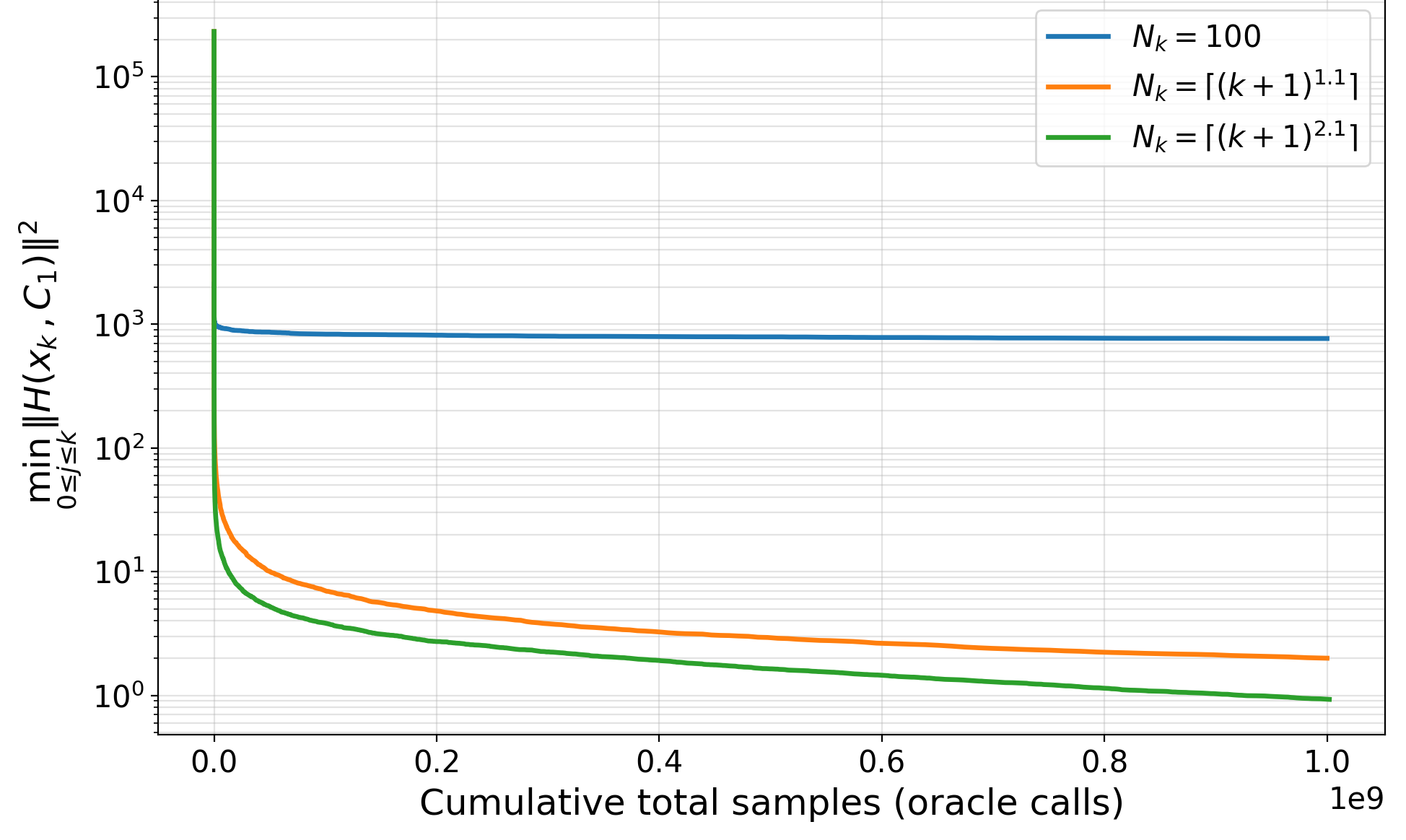}}
        \caption{Best residual \zi{$\min_{0\leq j\leq k}\|H(x_j,\bar \eta)\|^2$ with $\bar\eta=C_1$} versus the cumulative number of stochastic oracle calls.}
    \label{fig:MS_oracle_gap}
\end{figure}
Figure~\ref{fig:MS_pointwise_gap} reports the pointwise Monte Carlo estimate of \zi{$\mathbb{E}[\|H(x_k,\bar \eta)\|^2]$}. This plot illustrates the raw behavior of the residual map along the generated iterates and shows how the residual responds to different noise levels and batch-size schedules, together with $95\%$ confidence intervals. Since the theoretical result is stated for the best residual up to iteration $k$, Figure~\ref{fig:MS_best_gap} reports the sample average of \zi{$\min_{0\leq j\leq k}\|H(x_j,\bar \eta)\|^2$} over the simulation runs. This quantity is the empirical counterpart of the residual controlled by our convergence-rate result. Figure~\ref{fig:MS_oracle_gap} plots the same best residual against the cumulative number of stochastic oracle calls, highlighting the tradeoff between residual reduction and sampling effort under different batch-size choices. The constant-batch baseline illustrates that controlling the sampling error is essential in the stochastic setting, while the increasing batch-size schemes continue to reduce the best residual as the oracle budget grows.
Together, these plots illustrate the behavior predicted by the theory on monotone instances that are not covered by co-coercive SIVI methods. 

   


\section{Conclusions}
We studied stochastic inverse variational inequalities under monotonicity and Lipschitz continuity, without requiring co-coercivity or strong monotonicity of the mean operator. We proposed a regularized variance-controlled inverse projected-gradient method that combines Tikhonov regularization with an increasing batch-size sampling scheme. The proposed method was shown to generate almost surely bounded iterates whose distance to the SIVI solution set converges to zero almost surely. We also established an explicit $\mathcal O(T^{-1/2})$ nonasymptotic rate for the expected squared residual, together with the corresponding iteration and stochastic oracle complexity guarantees. A deterministic counterpart was analyzed and shown to achieve the same iteration complexity using exact operator evaluations. Numerical experiments on monotone network equilibrium problems that are not co-coercive illustrated the practical behavior of the method. Future work includes incorporating variance-reduction techniques that may improve stochastic oracle complexity, extending the framework to distributed SIVIs, and deriving sharper rates under additional structural assumptions.

\appendix\normalsize  
\section*{Appendix A: Proof of Lemma \ref{exp bound}}\label{app:proof-lem}
Using the update rule of $x_{k+1}$ and recalling the definition of $\bar w_{k,N_k}$ in \eqref{def:w} we obtain
\begin{align*}
    \| x_{k+1} - x_{\lambda_{k+1}}\|^2 &=   \| x_{k}-\theta_k\left(\ F_{\lambda_k}(x_k)+\bar w_{k,N_k} -z_k\right) - x_{\lambda_{k+1}}\|^2\\
    &= \| x_{k} - x_{\lambda_{k}}\|^2 + \|  x_{\lambda_{k}} - x_{\lambda_{k+1}}\|^2\\ &\quad+ \theta_k^2 \|z_k - (F_{\lambda_k}(x_k)+ \bar w_{k,N_k}) \|^2+2 \langle x_k -x_{\lambda_{k}}, x_{\lambda_{k}}-x_{\lambda_{k+1}}\rangle  \\&\quad + 2\theta_k \langle  x_k -x_{\lambda_{k}}, z_k- F_{\lambda_k}(x_k)-\bar w_{k,N_k} \rangle \\ &\quad+2\theta_k \langle x_{\lambda_{k}}-x_{\lambda_{k+1}},  z_k- (F_{\lambda_k}(x_k)+\bar w_{k,N_k} )\rangle, 
\end{align*}
where the second equality is obtained by using   $(a+b+c)^2 = a^2 + b^2 + c^2 + 2(ab + bc + ac)$. Next, adding and subtracting  $\bar z_k = \mathbf{P}_{X}\left[ F_{\lambda_k}(x_k) -\eta_{k} x_k\right] $ followed by using the Cauchy-Schwarz inequality and bounding $\|  x_{\lambda_{k}} - x_{\lambda_{k+1}}\|^2$ with Lemma \ref{M bound} then 
\begin{align*}
   & \| x_{k+1} - x_{\lambda_{k+1}}\|^2\\&\leq\| x_{k} - x_{\lambda_{k}}\|^2 + \tfrac{M^2|  \lambda_{k} - \lambda_{k+1}|^2}{\lambda_{k}^2} + \theta_k^2 \|z_k - \bar z_k+\bar z_k- (F_{\lambda_k}(x_k)+\bar w_{k,N_k})\|^2  \\&\quad+2 \tfrac{M|  \lambda_{k} - \lambda_{k+1}|}{\lambda_{k}}\|x_{k}-x_{\lambda_{k}}\| + 2\theta_k \langle  x_k -x_{\lambda_{k}}, \bar z_k- F_{\lambda_k}(x_k)\rangle \\&\quad+   2\theta_k \langle  x_k -x_{\lambda_{k}}, z_k- \bar z_k -\bar w_{k,N_k} \rangle\\&\quad + \tfrac{2\theta_k M|  \lambda_{k} - \lambda_{k+1}|}{\lambda_{k}}\|z_k - \bar z_k+\bar z_k-(F_{\lambda_k}(x_k)+\bar w_{k,N_k})\|\\
    &\leq \|x_{k} - x_{\lambda_{k}}\|^2 + \tfrac{M^2|  \lambda_{k} - \lambda_{k+1}|^2}{\lambda_{k}^2}\\&\quad+ 2\theta_k^2 (\|z_k - \bar z_k-\bar w_{k,N_k}\|^2 +\|\bar z_k-F_{\lambda_k}(x_k)\|^2)\\&\quad +2 \tfrac{M|  \lambda_{k} - \lambda_{k+1}|}{\lambda_{k}}\|x_{k}-x_{\lambda_{k}}\| + 2\theta_k \langle  x_k -x_{\lambda_{k}}, \bar z_k- F_{\lambda_k}(x_k)\rangle\\&\quad+   2\theta_k \langle  x_k -x_{\lambda_{k}}, z_k- \bar z_k -\bar w_{k,N_k} \rangle\\&\quad + \tfrac{2\theta_k M|  \lambda_{k} - \lambda_{k+1}|}{\lambda_{k}}(\|z_k - \bar z_k-\bar w_{k,N_k}\| +\|\bar z_k-F_{\lambda_k}(x_k)\|),
\end{align*}
where the last inequality is obtained by using the triangle inequality, and the fact that $(a+b)^2 \leq 2(a^2+b^2)$.
On the other hand, since condition of Lemma \ref{bound of inner pro} holds, $\lambda_k=(k+1)^{-p}\le 1$, $q\ge p$, and $C_1=(L_1+L_2)^2/\mu$, which imply $\frac{(L_1+\lambda_kL_2)^2}{2\lambda_k\mu}\le \frac{C_1}{2}(k+1)^p\le \frac{C_1}{2}(k+1)^q<\eta_k$, we have: 
\begin{align*}
    \langle  x_k -x_{\lambda_{k}}, \bar z_k- F_{\lambda_k}(x_k)\rangle &\leq -(\lambda_k \mu - \tfrac{(L_1+\lambda_kL_2)^2}{2\eta_k} )\| x_k -x_{\lambda_{k}}\|^2 \\&\quad-\tfrac{1}{2\eta_k}\|\bar  z_k- F_{\lambda_k}(x_k)\|^2.
\end{align*}
Combine these inequalities: 
\begin{align}\label{before youngs}
  \nonumber  &\| x_{k+1} - x_{\lambda_{k+1}}\|^2\\\nonumber&\leq \|x_{k} - x_{\lambda_{k}}\|^2 + \tfrac{M^2|  \lambda_{k} - \lambda_{k+1}|^2}{\lambda_{k}^2}+ 2\theta_k^2 (\|z_k - \bar z_k-\bar w_{k,N_k}\|^2 \\ \nonumber&\quad+\|\bar z_k-F_{\lambda_k}(x_k)\|^2) +2 \tfrac{M|  \lambda_{k} - \lambda_{k+1}|}{\lambda_{k}}\|x_{k}-x_{\lambda_{k}}\| \\ \nonumber&\quad+    2\theta_k \langle  x_k -x_{\lambda_{k}}, z_k- \bar z_k -\bar w_{k,N_k} \rangle-\tfrac{\theta_k}{\eta_k}\|\bar  z_k- F_{\lambda_k}(x_k)\|^2 \\ \nonumber&\quad+ \tfrac{2\theta_k M|  \lambda_{k} - \lambda_{k+1}|}{\lambda_{k}}(\|z_k - \bar z_k-\bar w_{k,N_k}\| +\|\bar z_k-F_{\lambda_k}(x_k)\|)  \\ \nonumber&\quad-2\theta_k (\lambda_k \mu - \tfrac{(L_1+\lambda_kL_2)^2}{2\eta_k} )\| x_k -x_{\lambda_{k}}\|^2 \\
  \nonumber &\leq (1-2\theta_k (\lambda_k \mu - \tfrac{(L_1+\lambda_kL_2)^2}{2\eta_k} ))\|x_{k} - x_{\lambda_{k}}\|^2 + \tfrac{M^2|  \lambda_{k} - \lambda_{k+1}|^2}{\lambda_{k}^2} \\ \nonumber&\quad-\theta_k (\tfrac{1}{\eta_k}-2\theta_k) \|\bar z_k-F_{\lambda_k}(x_k)\|^2 +2 \tfrac{M|  \lambda_{k} - \lambda_{k+1}|}{\lambda_{k}}\|x_{k}-x_{\lambda_{k}}\| \\ \nonumber&\quad+ \tfrac{2\theta_k M|  \lambda_{k} - \lambda_{k+1}|}{\lambda_{k}}\|\bar z_k-F_{\lambda_k}(x_k)\|+    2\theta_k \langle  x_k -x_{\lambda_{k}}, z_k- \bar z_k  \rangle \\ \nonumber&\quad+ \tfrac{4\theta_k M|  \lambda_{k} - \lambda_{k+1}|}{\lambda_{k}}\|\bar w_{k,N_k}\| + 8\theta_k^2 \|\bar w_{k,N_k}\|^2 -2\theta_k \langle  x_k -x_{\lambda_{k}},\bar w_{k,N_k}\rangle\\
  \nonumber&\leq (1-2\theta_k (\lambda_k \mu - \tfrac{(L_1+\lambda_kL_2)^2}{2\eta_k} ))\|x_{k} - x_{\lambda_{k}}\|^2 + \tfrac{M^2|  \lambda_{k} - \lambda_{k+1}|^2}{\lambda_{k}^2}\\ \nonumber &\quad-\theta_k (\tfrac{1}{\eta_k}-2\theta_k) \|\bar z_k-F_{\lambda_k}(x_k)\|^2 +2 \tfrac{M|  \lambda_{k} - \lambda_{k+1}|}{\lambda_{k}}\|x_{k}-x_{\lambda_{k}}\| \\ \nonumber &\quad+ \tfrac{2\theta_k M|  \lambda_{k} - \lambda_{k+1}|}{\lambda_{k}}\|\bar z_k-F_{\lambda_k}(x_k)\|+    2\theta_k \|x_k -x_{\lambda_{k}}\| \|\bar w_{k,N_k} \| \\  &\quad+ \tfrac{4\theta_k M|  \lambda_{k} - \lambda_{k+1}|}{\lambda_{k}}\|\bar w_{k,N_k}\| + 8\theta_k^2 \|\bar w_{k,N_k}\|^2 -2\theta_k \langle  x_k -x_{\lambda_{k}},\bar w_{k,N_k}\rangle,
\end{align}
where the second inequality is followed by the triangle inequality and nonexpansivity of projection, which leads to $\|z_k-\bar z_k - \bar w_{k,N_k}\| \leq \|z_k-\bar z_k\| + \| \bar w_{k,N_k}\| \leq 2 \|\bar w_{k,N_k}\|$. 
Furthermore, implementing Young's inequality $ab \leq \tfrac{a^2}{2\tau}+\tfrac{\tau b^2}{2}$ then we  have 
\begin{align}
    \nonumber  2\tfrac{M|  \lambda_{k} - \lambda_{k+1}|}{\lambda_{k}}\|x_{{k}}-x_{\lambda_{k}}\| &\leq  \tfrac{2M^2|  \lambda_{k} - \lambda_{k+1}|^2}{\lambda_{k}^3 \theta_k \mu} + \tfrac{\theta_k \lambda_k \mu}{2}\| x_{k} - x_{\lambda_{k}}\|^2,
    \\
    \nonumber \tfrac{ M|  \lambda_{k} - \lambda_{k+1}|}{\lambda_{k}}\|\bar z_k- F_{\lambda_k}(x_k)\|  &\leq  \tfrac{ \eta_k M^2|  \lambda_{k} - \lambda_{k+1}|^2}{\lambda_{k}^2} + \tfrac{\|\bar z_k- F_{\lambda_k}(x_k)\|^2}{4\eta_k},
    \\ 
    \nonumber \|x_k -x_{\lambda_{k}}\| \|\bar w_{k,N_k} \| &\leq \tfrac{\beta_k \|x_k -x_{\lambda_{k}}\|^2}{2} + \tfrac{ \|\bar w_{k,N_k} \|^2}{2 \beta_k}, 
\end{align}
with Young's parameters $ \tfrac{\theta_k \lambda_k \mu}{2}$, $ \tfrac{1}{2\eta_k} $ and $\tfrac{1}{\beta_k}$ , respectively.  Combining these bounds in \eqref{before youngs} one can obtain: 
\begin{align*}
        \| x_{k+1} - x_{\lambda_{k+1}}\|^2&\leq (1-2\theta_k (\lambda_k \mu - \tfrac{(L_1+\lambda_kL_2)^2}{2\eta_k} )+\tfrac{\theta_k \lambda_k \mu}{2})\| x_{k} - x_{\lambda_{k}}\|^2 \\
    &\quad+ \tfrac{M^2|  \lambda_{k} - \lambda_{k+1}|^2}{\lambda_{k}^2} - \theta_k(\tfrac{1}{2\eta_k}-2\theta_k)\|\bar z_k- F_{\lambda_k}(x_k)\|^2 \\
    &\quad+\tfrac{2\theta_k \eta_k M^2|  \lambda_{k} - \lambda_{k+1}|^2}{\lambda_{k}^2} +2 \tfrac{M^2|  \lambda_{k} - \lambda_{k+1}|^2}{\lambda_{k}^3 \theta_k \mu}+8\theta_k^2\|\bar w_{k,N_k}\|^2 \\
    &\quad-  2\theta_k \langle  x_k -x_{\lambda_{k}},\bar w_{k,N_k} \rangle + \theta_k\beta_k\|x_k-x_{\lambda_k}\|^2+\tfrac{\theta_k}{\beta_k}\|\bar w_{k,N_k}\|^2\\
    &\quad+\tfrac{4\theta_k M|  \lambda_{k} - \lambda_{k+1}|}{\lambda_{k}}\|\bar w_{k,N_k}\|.
\end{align*} 
 Rearranging terms and using the fact that  $ \tfrac{ 4M|  \lambda_{k} - \lambda_{k+1}|}{\lambda_{k}}\|\theta_k \bar w_{k,N_k}\| \leq2( \tfrac{ M^2|  \lambda_{k} - \lambda_{k+1}|^2}{\lambda_{k}^2} +\theta_k^2 \| \bar w_{k,N_k}\|^2)$ yields
\begin{align*}\nonumber
  \| x_{k+1} - x_{\lambda_{k+1}}\|^2 &\leq (1-\theta_k (\tfrac{3}{2}\lambda_k \mu - \tfrac{(L_1+\lambda_kL_2)^2}{\eta_k}-\beta_k ))\| x_{k} - x_{\lambda_{k}}\|^2 \\\nonumber&\quad- \theta_k(\tfrac{1}{2\eta_k}-2\theta_k)\|\bar z_k- F_{\lambda_k}(x_k)\|^2  + \tfrac{3M^2|  \lambda_{k} - \lambda_{k+1}|^2}{\lambda_{k}^2} \\\nonumber&\quad+\tfrac{2\theta_k \eta_k M^2|  \lambda_{k} - \lambda_{k+1}|^2}{\lambda_{k}^2} +2 \tfrac{M^2|  \lambda_{k} - \lambda_{k+1}|^2}{\lambda_{k}^3 \theta_k \mu}+10\theta_k^2\|\bar w_{k,N_k}\|^2\\
    &\quad -  2\theta_k \langle  x_k -x_{\lambda_{k}},\bar w_{k,N_k} \rangle +\tfrac{\theta_k}{\beta_k}\|\bar w_{k,N_k}\|^2.
\end{align*} 
Choose $\lambda_k = (k+1)^{-p}$, $\eta_k = C_1 (k+1)^{q}$, $\beta_k = \tfrac{\lambda_k \mu}{4}$ , $\theta_k = C_2 (k+1)^{-r}$, where $C_1 =  \frac{(L_1 + L_2)^2}{\mu} $ and $C_2 = \tfrac{1}{8C_1}$.  
Since $0<p\leq q<1$ and $\lambda_k=(k+1)^{-p}$, we have $
\frac{(L_1+\lambda_kL_2)^2}{\eta_k}
= \mu (k+1)^{-q}
\leq \mu (k+1)^{-p}
= \lambda_k\mu.$
Hence,
$1-\frac{5}{4}\theta_k\lambda_k\mu
+\frac{\theta_k(L_1+\lambda_kL_2)^2}{\eta_k}
\leq
1-\frac{1}{4}\theta_k\lambda_k\mu
<1.$
This leads to \eqref{equ for main th0} and \eqref{equ for main th}. 
Taking conditional expectation and using Assumption \ref{assump_error}, we have the desired result.\qed

\section*{Appendix B: Proof of Theorem \ref{main th}}\label{app:proof-thm}
Using \eqref{equ for main th} in Lemma \ref{exp bound}, rearranging terms one can obtain: 
\begin{align*}
   &\theta_k(\tfrac{1}{2\eta_k}-2\theta_k)\|\bar z_k- F_{\lambda_k}(x_k)\|^2  \\&\leq  \| x_{k} - x_{\lambda_{k}}\|^2 -\| x_{k+1} - x_{\lambda_{k+1}}\|^2+ \tfrac{3M^2|  \lambda_{k} - \lambda_{k+1}|^2}{\lambda_{k}^2} +\tfrac{2\theta_k \eta_k M^2|  \lambda_{k} - \lambda_{k+1}|^2}{\lambda_{k}^2} \\&\quad+2 \tfrac{M^2|  \lambda_{k} - \lambda_{k+1}|^2}{\lambda_{k}^3 \theta_k \mu}+10\theta_k^2\|\bar w_{k,N_k}\|^2 \zi{-}  2\theta_k \langle  x_k -x_{\lambda_{k}},\bar w_{k,N_k} \rangle +\tfrac{4\theta_k}{\lambda_k \mu}\|\bar w_{k,N_k}\|^2.
\end{align*} 
 Using $\eta_k = C_1(k+1)^{q}$, $\theta_k = C_2(k+1)^{-r}$, $C_2 = \frac{1}{8C_1}$ and $0 < q \leq r$  we have $1-4\eta_k\theta_k = 1-\tfrac{1}{2(k+1)^{r-q}} \geq 1-\tfrac{1}{2} > 0$ which follows from the fact that $\tfrac{1}{2(k+1)^{r-q} } \leq \tfrac{1}{2}$ for all $k\geq 0$.  Next, dividing both sides by $\theta_k(\tfrac{1}{2\eta_k}-2\theta_k) $ we have: 

\begin{align*}
  \|\bar z_k- F_{\lambda_k}(x_k)\|^2  &\leq \tfrac{2\eta_k}{\theta_k(1-4\theta_k\eta_k)}\big( \| x_{k} - x_{\lambda_{k}}\|^2 -\| x_{k+1} - x_{\lambda_{k+1}}\|^2+ \tfrac{3M^2|  \lambda_{k} - \lambda_{k+1}|^2}{\lambda_{k}^2} \\&\quad+\tfrac{2\theta_k \eta_k M^2|  \lambda_{k} - \lambda_{k+1}|^2}{\lambda_{k}^2} +2 \tfrac{M^2|  \lambda_{k} - \lambda_{k+1}|^2}{\lambda_{k}^3 \theta_k \mu}+10\theta_k^2\|\bar w_{k,N_k}\|^2\\
    &\quad \zi{-}  2\theta_k \langle  x_k -x_{\lambda_{k}},\bar w_{k,N_k} \rangle +\tfrac{4\theta_k}{\lambda_k \mu}\|\bar w_{k,N_k}\|^2\big)
\end{align*} 
Next, bound the left-hand side by defining   $ H_{\lambda_k}(x_k) \triangleq \bar z_k - F_{\lambda_k}(x_k) $  and knowing  $\|H(x,\eta)\|^2 \leq 2\|H_\lambda(x,\eta)\|^2 +{8\lambda^2}\|\Phi(x)\|^2$ : 
\begin{align*}
  \|H(x_k,\eta_k)\|^2   &\leq \tfrac{4\eta_k}{\theta_k(1-4\theta_k\eta_k)}\big( \| x_{k} - x_{\lambda_{k}}\|^2 -\| x_{k+1} - x_{\lambda_{k+1}}\|^2+ \tfrac{3M^2|  \lambda_{k} - \lambda_{k+1}|^2}{\lambda_{k}^2} \\&\quad+\tfrac{2\theta_k \eta_k M^2|  \lambda_{k} - \lambda_{k+1}|^2}{\lambda_{k}^2} +2 \tfrac{M^2|  \lambda_{k} - \lambda_{k+1}|^2}{\lambda_{k}^3 \theta_k \mu}+10\theta_k^2\|\bar w_{k,N_k}\|^2\\
    &\quad \zi{-}  2\theta_k \langle  x_k -x_{\lambda_{k}},\bar w_{k,N_k} \rangle +\tfrac{4\theta_k}{\lambda_k \mu}\|\bar w_{k,N_k}\|^2\big)+{8\lambda_k^2} \|\Phi(x_k)\|^2.
\end{align*} 
Now summing over $k=0, \dots ,T-1$ then we have : 
\begin{align*}
 \sum_{k=0}^{T-1} \|H(x_k,\eta_k)\|^2   &\leq \sum_{k=0}^{T-1}\tfrac{4\eta_k}{\theta_k(1-4\theta_k\eta_k)}\big( \| x_{k} - x_{\lambda_{k}}\|^2 -\| x_{k+1} - x_{\lambda_{k+1}}\|^2\\&\quad+ \tfrac{3M^2|  \lambda_{k} - \lambda_{k+1}|^2}{\lambda_{k}^2} +\tfrac{2\theta_k \eta_k M^2|  \lambda_{k} - \lambda_{k+1}|^2}{\lambda_{k}^2} +2 \tfrac{M^2|  \lambda_{k} - \lambda_{k+1}|^2}{\lambda_{k}^3 \theta_k \mu}\\&\quad+10\theta_k^2\|\bar w_{k,N_k}\|^2 \zi{-}  2\theta_k \langle  x_k -x_{\lambda_{k}},\bar w_{k,N_k} \rangle +\tfrac{4\theta_k}{\lambda_k \mu}\|\bar w_{k,N_k}\|^2\big)\\&\quad+\sum_{k=0}^{T-1}{8\lambda_k^2} \|\Phi(x_k)\|^2.
\end{align*} 
Consider $\Gamma_k= \tfrac{\eta_k}{ \theta_k(1-4\theta_k \eta_k )}$ and take the total expectation $\mathbb{E}[\cdot]$ on both sides. By Assumption \ref{assump_error}, we have $\mathbb E[\bar w_{k,N_k}\mid \mathcal F_k]=0$ and $\mathbb E[\|\bar w_{k,N_k}\|^2]\leq \tfrac{\nu^2}{N_k}$. Thus, we obtain
\begin{align*}
\sum_{k=0}^{T-1} \mathbb{E}[\|H(x_k,\eta_k)\|^2] 
&\leq \mathbb{E}\left[\sum_{k=0}^{T-1}4\Gamma_k\left( \| x_{k} - x_{\lambda_{k}}\|^2 -\| x_{k+1} - x_{\lambda_{k+1}}\|^2\right)\right] \\&\quad+\sum_{k=0}^{T-1} 8\lambda_k^2 \mathbb{E}[\|\Phi(x_k)\|^2] \\&\quad+\sum_{k=0}^{T-1} \tfrac{4\Gamma_k M^2 |\lambda_k - \lambda_{k+1}|^2}{\lambda_k^2} \left( 3 + 2\theta_k \eta_k + \tfrac{2}{\lambda_k \theta_k \mu} \right) \\&\quad+ \sum_{k=0}^{T-1} \tfrac{4\Gamma_k \nu^2 }{ N_k} \left( 10\theta_k^2 + \tfrac{4\theta_k}{\lambda_k \mu} \right).
\end{align*}
Since $\Phi$ is Lipschitz continuous, we have
$\|\Phi(x_k)\|^2 \leq 2\|\Phi(0)\|^2 + 2L_2^2\|x_k\|^2$.
Moreover, $\{x_{\lambda_k}\}$ is bounded and 
$\sup_{k\geq 0}\mathbb{E}[\|x_k-x_{\lambda_k}\|^2]<\infty$, which implies
$\sup_{k\geq 0}\mathbb{E}[\|x_k\|^2]<\infty$.
Therefore, there exists a deterministic constant $C>0$ such that
$\sup_{k\geq 0}\mathbb{E}[\|\Phi(x_k)\|^2]\leq C$.
Therefore,
$\sum_{k=0}^{T-1} 8\lambda_k^2 \mathbb{E}[\|\Phi(x_k)\|^2]\leq\sum_{k=0}^{T-1} 8\lambda_k^2 C.$

On the other hand, 
\begin{align*} &\mathbb{E}[\sum_{k=0}^{T-1}\Gamma_k\big(\| x_{k} - x_{\lambda_{k}}\|^2 - \| x_{k+1} - x_{\lambda_{k+1}}\|^2 \big)] \\&= \mathbb{E}[\Gamma_0 \|x_0 - x_{\lambda_0}\|^2 - \Gamma_{T-1} \|x_T - x_{\lambda_T}\|^2 + \sum_{k=1}^{T-1} (\Gamma_k - \Gamma_{k-1}) \|x_k - x_{\lambda_k}\|^2 \ ]\\&\leq \mathbb{E}[\Gamma_0 \|x_0 - x_{\lambda_0}\|^2 + \sum_{k=1}^{T-1} (\Gamma_k - \Gamma_{k-1}) \|x_k - x_{\lambda_k}\|^2]\\
&\leq D^2 \Gamma_{T-1},\end{align*}

where the last inequality follows from the fact that $\Gamma_k$ is a non-decreasing sequence, followed by  Lemma \ref{bound x_k} which shows there exists a constant $D > 0$ such that $\mathbb{E}[\|x_k - x_{\lambda_k}\|^2] \leq D^2$ for all $k$ , where $N_k =\lceil(k+1)^{1+\delta}\rceil$. 
So, we have:
\begin{align*}
\sum_{k=0}^{T-1} \mathbb{E}[\|H(x_k,\eta_k)\|^2] &\leq \sum_{k=0}^{T-1} \tfrac{4\Gamma_k M^2 |\lambda_k - \lambda_{k+1}|^2}{\lambda_k^2} \left( 3 + 2\theta_k \eta_k + \tfrac{2}{\lambda_k \theta_k \mu} \right) \
 \\ &\quad+ \sum_{k=0}^{T-1} \tfrac{4\Gamma_k \nu^2 }{ N_k} \left( 10\theta_k^2 + \tfrac{4\theta_k}{\lambda_k \mu} \right) +4\Gamma_{T-1} D^2+ \sum_{k=0}^{T-1} 8\lambda_k^2 C.
\end{align*}

 Based on the chosen sequences $\lambda_k = (k+1)^{-p}$, $\eta_k = C_1(k+1)^{q}$, and $\theta_k = C_2(k+1)^{-r}$, where $C_1 = \frac{(L_1 + L_2)^2}{\mu}$ and $C_2 = \frac{1}{8C_1}$, we observe that $4\theta_k \eta_k = \frac{1}{2}(k+1)^{q-r} \leq \frac{1}{2}$, which implies $\Gamma_k = \frac{\eta_k}{\theta_k(1-4\theta_k \eta_k)} \leq 2\frac{\eta_k}{\theta_k} = 16C_1^2(k+1)^{q+r}$.
Next, using Lemma \ref{MVT} with $h(k) = (k+1)^{-p}$ to bound $|\lambda_k - \lambda_{k+1}|^2$ one can obtain $|\lambda_k - \lambda_{k+1}|^2 \leq p^2(k+1)^{-2-2p}.$
Using in the previous inequality, yields: 
\begin{align*}
&\sum_{k=0}^{T-1} \mathbb{E}[\|H(x_k,\eta_k)\|^2] \\&\leq (8 p C_1 M)^2 \sum_{k=0}^{T-1} (k+1)^{-2+r+q} \left( \tfrac{16 C_1 (k+1)^{p+r}}{\mu} + \tfrac{(k+1)^{q-r}}{4} + 3 \right) \\&\quad
 + 8C_1\nu^2  \sum_{k=0}^{T-1} \tfrac{(k+1)^{r+q}}{N_k} \left( 10 C_2 (k+1)^{-2r} + \tfrac{4 (k+1)^{-r+p}}{\mu} \right) \\
&\quad + \sum_{k=0}^{T-1} 8(k+1)^{-2p} C + 64 C_1^2 D^2 T^{q+r}.
\end{align*}
Now by considering $N_k = \lceil(k+1)^{1+\delta}\rceil$ and providing a lower-bound for the left-hand-side of the above inequality by $\min_k \mathbb E[\|H(x_k,\eta_k)\|^2] $ then we have

\begin{align*}
\nonumber &T\min_k \mathbb E[\|H(x_k,\eta_k)\|^2] \\& \leq (8 p C_1 M)^2 \sum_{k=0}^{T-1} (k+1)^{-2+r+q} \left( \tfrac{16 C_1 (k+1)^{p+r}}{\mu} + \tfrac{(k+1)^{q-r}}{4} + 3 \right) \\&\nonumber\quad
 + 8C_1\nu^2  \sum_{k=0}^{T-1} {(k+1)^{r+q-1-\delta}} \left( 10 C_2 (k+1)^{-2r} + \tfrac{4 (k+1)^{-r+p}}{\mu} \right) \\
&\quad + \sum_{k=0}^{T-1} 8(k+1)^{-2p} C + 64 C_1^2 D^2 T^{q+r}.
\end{align*}
Finally, dividing both sides by $T$ yields the desired result.
\qed

\bibliographystyle{siam}
\bibliography{bilbo}
\end{document}